\documentclass[a4paper]{article}
\usepackage{mathptmx}
\usepackage{amscd,amssymb,amsmath,amsthm}
\usepackage{mathrsfs}

\makeatletter
\renewcommand{\@seccntformat}[1]
{{\csname the#1\endcsname}.\hspace{0.3em}}

\renewcommand{\section}{\@startsection
{section}
{1}
{0mm}
{-1.5\baselineskip}
{\baselineskip}
{\bfseries\normalsize}}

\renewcommand{\subsection}{\@startsection
{subsection}
{2}
{0mm}
{-\baselineskip}
{0.5\baselineskip}
{\normalsize\itshape}}

\makeatother

\theoremstyle{plain}
\newtheorem{theorem}{Theorem}[section]
\newtheorem{lemma}[theorem]{Lemma}
\newtheorem{corollary}[theorem]{Corollary}
\newtheorem{prop}[theorem]{Proposition}
\newtheorem*{theorem*}{Theorem}
\newtheorem*{corollary*}{Corollary}

\newtheorem{claim}{Claim}
\newtheorem*{abel1}{Abelian subspaces theorem I}
\newtheorem*{abel2}{Abelian subspaces theorem II}
\newtheorem*{FT}{Factorisation theorem}
\newtheorem*{SR}{Strong rigidity}
\newtheorem*{VC}{Vaisman's conjecture}

\theoremstyle{definition}
\newtheorem*{defin*}{Definition}

\theoremstyle{remark}
\newtheorem{example}{Example}[section]
\newtheorem*{remark*}{Remark}

\DeclareMathAlphabet{\matheur}{U}{eur}{m}{n}
\DeclareMathAlphabet{\matheus}{U}{eus}{m}{n}
\DeclareMathAlphabet{\matheuf}{U}{euf}{m}{n}

\numberwithin{equation}{section}
\newcommand{\abs}[1]{\left\lvert#1\right\rvert}
\newcommand{\norm}[1]{\left\lVert#1\right\rVert}

\newcommand{\pardif}[2]{\frac{\partial #1}{\partial #2}}

\DeclareMathOperator{\Hom}{Hom}

\DeclareMathOperator{\rank}{rank}

\begin{document}

\author{Gerasim  Kokarev
\\ {\small\it School of Mathematics, The University of Edinburgh}
\\ {\small\it King's Buildings, Mayfield Road, Edinburgh EH9 3JZ, UK}
\\ {\small\it Email: {\tt G.Kokarev@ed.ac.uk}}
}

\title{On pseudo-harmonic maps in conformal geometry}
\date{}
\maketitle

\begin{abstract}
We extend harmonic map techniques to the setting of more general
differential equations in conformal geometry. We discuss existence
theorems and obtain an extension of Siu's strong rigidity to
K\"ahler-Weyl geometry. Other applications include topological
obstructions to the existence of K\"ahler-Weyl structures. For
example, we show that no co-compact lattice in $SO(1,n)$, $n>2$, can
be the fundamental group of a compact K\"ahler-Weyl manifold of
certain type.
\end{abstract}


\section{Introduction}
The purpose of this paper is to introduce and study an elliptic
quasilinear system of equations on maps between manifolds endowed with
linear connections. This system generalises the harmonic map equation
and, in many situations, is more suitable for geometric
applications. We demonstrate this in the context of conformal
geometry. More precisely, the first principal application is the
following extension of Siu's strong rigidity (Theorem~\ref{Siu2}).
\begin{SR}
Let $M$ be a compact K\"ahler-Weyl manifold and $M'$ be a compact
locally Hermitian symmetric space of non-compact type whose universal
cover does not contain the hyperbolic plane as a factor. Suppose that
there exists a homotopy equivalence $u:M\to M'$. If $M$ has complex
dimension $2$ or admits a pluricanonical metric, then $u$ is homotopic
to a biholomorphism for some invariant complex structure on $M'$.
\end{SR}

Any K\"ahler metric on $M$ is pluricanonical (see Sect.~\ref{PLURI}
for a precise definition) and for this case the strong rigidity is
due to Siu~\cite{Siu}. In complex dimension two any compact complex
manifold admits a K\"ahler-Weyl structure and the statement shows that
the complex structure on compact quotients of the unit ball in
$\mathbf C^2$ is globally rigid among all possible complex structures;
see~\cite{JY,JY2}. In complex dimension greater than two the
K\"ahler-Weyl condition forces $M$ to be locally conformal K\"ahler;
see Sect.~\ref{PLURI}. Such manifolds form a significantly larger
class of complex manifolds than K\"ahler ones. The theorem above
states that the complex structure on the locally Hermitian symmetric
space $M'$ is globally rigid among certain locally conformal K\"ahler
variations.

The difference between this version of the strong rigidity and the
version for K\"ahler manifolds can be illustrated by the following
conjecture, see~\cite[Ch.2]{DO}.
\begin{VC}
A compact locally conformal K\"ahler manifold of the ``same topology''
as a compact K\"ahler manifold admits a global K\"ahler metric.
\end{VC} 
In our context the appropriate meaning for the phrase ``same
topology'' is ``same homotopy type''. The positive answer to this
conjecture would imply the strong rigidity among all locally conformal
K\"ahler complex structures as a consequence of the rigidity
among K\"ahler ones. However, almost nothing is known on Vaisman's
conjecture in complex dimension greater than two.

As an application of the theorem above, we partially confirm
this conjecture for manifolds with pluricanonical metrics.
\begin{corollary*}
Let $M$ be a compact locally conformal K\"ahler manifold of the same
homotopy type as a locally Hermitian symmetric space of non-compact
type whose universal cover does not contain the hyperbolic plane as a
factor. If $M$ admits a pluricanonical metric, then it admits a global
K\"ahler metric.
\end{corollary*}

The second principal application of our technique is the following
extension of the results by Carlson and Toledo~\cite{CT89,CT97};
Theorem~\ref{lattices}.
\begin{theorem*}
Let $\Gamma$ be a co-compact discrete subgroup of $SO(1,n)$ with
$n>2$. If $\Gamma$ is the fundamental group of a K\"ahler-Weyl
manifold, then the latter can not be a complex surface and can not
admit a pluricanonical metric. 
\end{theorem*}

Mention that the result of Taubes~\cite{Taubes} implies that any
finitely presentable group is the fundamental group of a closed
complex $3$-dimensional manifold, cf.~\cite{Am}. The groups in the
theorem illustrate the topological difference between the class of all
complex manifolds (in dimension greater than two) and K\"ahler-Weyl
manifolds with pluricanonical metrics. 

Now we outline the organisation of the paper. In
Sect.~\ref{sect1}-\ref{1:3} we introduce main equations and discuss
existence theorems. The study of equations of this kind was started in
the paper~\cite{JY2} by Jost and Yau, where the authors consider a
similar equation in Hermitian geometry. In particular, the main
existence result (Theorem~\ref{ExTh}) is a version of the one
in~\cite{JY2}. Its proof is discussed in Sect.~\ref{ETproof}.

In Sect.~\ref{firstBT}, as a first application of the existence
theorem, we obtain topological constraints on Weyl manifods with
non-negative symmetric part of the Ricci-Weyl
curvature. Sect.~\ref{PLURI}-\ref{top} are devoted to the applications
to K\"ahler-Weyl geometry. These include the strong rigidity and
topological obstructions to the existence of certain K\"ahler-Weyl
structures. The results generalise known ones for K\"ahler manifolds
and complex surfaces and put the latter into a general picture of
K\"ahler-Weyl geometry.

Some of our results admit further extensions. For example, one can
consider the strong rigidity for irreducible quotients of
polydisks. This would imply that the hypothesis on the universal cover
is unnecessary in the corollary above. The existence problem can be
also studied in the more general setting of twisted pseudo-harmonic maps.

\section{Preliminaries on pseudo-harmonic maps}
\label{sect1}
Recall the definition of the harmonic map equation. Let $(M,g)$ and
$(M',g')$ be Riemannian manifolds of dimensions $n$ and $n'$
respectively. Their Riemannian metrics give rise to a natural
metric on the $1$-jet bundle  $J^1(M,M')$ over the space $M\times M'$
and for maps $u:M\to M'$ we consider the energy functional 
$$
E(u)=\frac{1}{2}\int\limits_M\norm{du(x)}^2d\mathit{Vol}_g(x),\qquad x\in M.
$$
The Euler-Lagrange equation for this functional
$$
-\tau(u)(x)=0,\qquad x\in M,
$$
is called the {\it harmonic map equation} and its solutions are called
{\it harmonic mappings}. In local coordinates on $M$ and $M'$ the
operator $\tau(u)$ has the form
$$
\tau^i(u)=\Delta_Mu^i+g^{\alpha\beta}\Gamma^{\prime\ i}_{jk}
\pardif{u^j}{x^\alpha}\pardif{u^k}{x^\beta},
$$
where $g^{\alpha\beta}$ and $\Gamma^{\prime\ i}_{jk}$ denote the
tensor inverse to the metric on $M$ and the Cristoffel symbols of the
Levi-Civita connection on $M'$ respectively, and $\Delta_M$ is the 
Laplace-Beltrami operator on $M$.\footnote{There are different
conventions for the choice of the sign of the Laplace-Beltrami
operator. Due to our definition this operator is non-positive.}

The vector field $\tau(u)(x)$, where $x\in M$, is called the tension
field and can be alternatively described as follows. Consider the
second fundamental form $\mathcal D^2u$ of a map $u$, given by
\begin{equation}
\label{SFF}
\mathcal D^2u(X,Y)=\widetilde\nabla_X(du)(Y).
\end{equation}
Above $X$ and $Y$ are vector fields on $M$ and $\widetilde\nabla$
denotes the connection on the tensor product $T^*M\otimes u^*TM'$
induced by the Levi-Civita connections $\nabla$ and $\nabla'$ on $M$
and $M'$ respectively. It is a simple calculation to show that the
tension field $\tau(u)$ coincides with $\mathit{trace}_g{\mathcal
D^2}u$. This suggests to consider the following more general
operator; cf.~\cite{GK1}.
\begin{defin*}
Given arbitrary linear torsion-free connections $\nabla$ and $\nabla'$
on $M$ and $M'$ determine a connection $\widetilde\nabla$ on the
tensor product $T^*M\otimes u^*TM'$; we denote by the symbol
${\mathcal D^2}u$ the form given by~\eqref{SFF} with respect to this
connection. For a given Riemannian metric $g$ on $M$ the
correspondence 
$$
\mathit{Maps}(M,M')\ni u\longmapsto\mathit{trace}_g{\mathcal
  D^2}u\in\mathit{Sections}(u^*TM')
$$
defines a second order elliptic differential operator, called the {\it
  pseudo-harmonic map operator} and denoted below by
$\tau(g,\nabla,\nabla')$; the solutions of the corresponding equation 
\begin{equation}
\label{PHME}
\mathit{trace}_g{\mathcal D^2}u(x)=0,\qquad x\in M,
\end{equation}
are called {\it pseudo-harmonic mappings}. 
\end{defin*}
The reason for introducing the more general equation is that
geometrically interesting maps are often solutions of it. In more
detail, suppose that given manifolds $M$ and $M'$ are endowed with
some structures and connections $\nabla$ and $\nabla'$ which respect
(or preserve) these structures. Then we expect the morphisms of the
structures be corresponding pseudo-harmonic maps. For example, if $M$
and $M'$ are Riemannian manifolds and $\nabla$ and $\nabla'$ are their
Levi-Civita connections, then totally geodesic maps are, of course,
harmonic. Analogously, holomorphic and anti-holomorphic maps are
always  Hermitian harmonic (and not necessarily harmonic!) in the
example below.

\begin{example}[Hermitian harmonic map equation~\cite{JostYau}]
\label{HerHME}
Suppose manifolds $M$ and $M'$ are complex and let $\nabla$ and
$\nabla'$ be corresponding torsion-free complex connections. Let $g$
be a Hermitian metric on $M$.  A calculation shows that in local
coordinates $(1/2)\tau(g,\nabla,\nabla')$ has the form
$$
\frac{1}{2}\tau(g,\nabla,\nabla')(u)^i=g^{\alpha\bar\beta}\left(
\frac{\partial^2u^i}{\partial z^\alpha\partial z^{\bar\beta}}+
{\Gamma'}^i_{jk}\pardif{u^j}{z^\alpha}\pardif{u^k}
{z^{\bar\beta}}\right),
$$
where ${\Gamma'}^i_{jk}$ stand for Cristoffel symbols of the
connection on $M'$, see~\cite{GK1}. In particular, the equation
$(1/2)\tau(g,\nabla,\nabla')=0$ coincides with the Hermitian harmonic
map equation, introduced by Jost and Yau in~\cite{JostYau}; the
corresponding solutions are precisely harmonic maps when $M$ and $M'$
are K\"ahler.
\end{example}
\begin{example}[Hermitian harmonic maps into Riemannian manifolds]
\label{Nher}
Here we describe a natural pseudo-harmonic map equation on maps from a
complex manifold $(M,J)$ to a Riemannian manifold $(M',g')$. Let $g$
and $\nabla^\dagger$ be a Hermitian metric and a torsion-free complex
connection on $M$ and $\nabla'$ be the Levi-Civita connection of the
metric $g'$. The natural morphisms of manifolds $(M,J)$ and $(M',g')$
-- the so-called {\it pluriharmonic maps}, see
Sect.~\ref{PLURI}-\ref{top}, -- solve the corresponding
equation $\tau(g,\nabla^\dagger,\nabla')(u)=0$. We also refer to it
as the {\it Hermitian harmonic map equation}. A computation shows that
the latter differs from the harmonic map equation (for the metrics $g$
and $g'$) by a linear first order term,
$$
\tau(g,\nabla^\dagger,\nabla')(u)-\tau(u)=-du(J\!\cdot\!\delta J).
$$
Here the vector field $\delta J$ is given by the formula
$$
\delta J=-\mathit{trace}_g(\nabla J),
$$
where $\nabla$ denotes the Levi-Civita connection of the Hermitian
metric $g$. Recall that Hermitian manifolds with vanishing $\delta J$
are called co-symplectic. In particular, for such domains these
Hermitian harmonic maps coincide with harmonic maps.
\end{example}

Now we list a number of basic properties of pseudo-harmonic maps,
which are essentially consequences of the fact that they solve second
order elliptic differential equations. The statements below are
analogous to these by Sampson~\cite{Sam0} for harmonic maps; the
proofs in~\cite{Sam0} carry over without essential changes to the
pseudo-harmonic setting.
\begin{prop}[Unique continuation]
Let $M$ be a connected manifold and $u$, $v$ be two pseudo-harmonic
maps $M\to M'$, the solutions of equation~\eqref{PHME}. If they agree
on an open subset of $M$, then they are identical; the same conclusion
holds if $u$ and $v$ agree to infinitely high order at some point. In
particular, a pseudo-harmonic map which is constant on an open subset
is a constant map.
\end{prop}
Thus, if a pseudo-harmonic map $u$ has rank zero on an open subset of
$M$, i.e. is constant on an open subset, it must have rank zero
everywhere. In the case of real-analytic manifolds (endowed with
real-analytic linear connections $\nabla$ and $\nabla'$ and a
real-analytic metric $g$ on the domain manifold) the pseudo-harmonic
map is also real-analytic and it follows that, if it has rank $r$ on
an open subset of $M$, then it has rank $r$ on an open and dense
subset. In the differentiable case we have the following statement
when the rank equals one.
\begin{prop}
\label{rank<1}
Let $M$ be a connected manifold and $u:M\to M'$ be a pseudo-harmonic
map. If the differential $du$ has rank one on an open subset of $M$,
then $u$ maps $M$ into a geodesic arc (with respect to the connection
$\nabla'$) in the manifold $M'$ and $du$ has rank one on an open and
dense subset. If $M$ is closed, then the geodesic arc is closed.
\end{prop}
The proof of this proposition uses the following version of the
maximum principle.
\begin{prop}[Maximum principle]
Let $u:M\to M'$ be a pseudo-harmonic map and $p\in M$, $q\in M'$ be
points such that the former is mapped onto the latter under $u$,
$q=u(p)$. Let $S$ be a piece of a hypersurface in $M'$ passing through
$q$, at which we assume that the second fundamental form (with respect
to the connection on $M'$) is definite. If $u$ is not a constant
mapping, then no neighbourhood of $p$ is mapped entirely to the
concave side of $S$.
\end{prop}

The existence of pseudo-harmonic maps for arbitrary connections
$\nabla$ and $\nabla'$ seems to be a subtle question. To make a first
step we suppose that a {\em connection $\nabla'$ on the target
  manifold $M'$ is metric}. In this case the pseudo-harmonic map
equation differs from the harmonic map equation by a first order
term and we discuss, in Sect.~\ref{1:3}, existence results for such
pseudo-harmonic maps when $M'$ has non-positive sectional curvature.

\section{Pseudo-harmonic maps in Weyl geometry}
\label{1:3}
For the rest of the paper $\nabla$ and $\nabla'$ {\em always stand for
  the Levi-Civita connections of the metrics $g$ and $g'$ on $M$ and
  $M'$} respectively. We now specialise the considerations to the
  setting of Weyl geometry.

\subsection{First definitions}
A Weyl structure on a conformal manifold $(M,c)$ is a torsion-free
linear connection $\nabla^W$ preserving the conformal structure $c$.
This means that for any Riemannian metric $g\in c$ there exists an
$1$-form $\Theta$, called the {\em Higgs field}, such that
$\nabla^Wg=\Theta\otimes g$. Alternatively, one can define $\nabla^W$
by the following formula
\begin{equation}
\label{WeylDef}
\nabla^W_XY=\nabla_XY-\frac{1}{2}\Theta (X)Y-\frac{1}{2}\Theta (Y)X
+\frac{1}{2}g(X,Y)\Theta^\sharp,
\end{equation}
where $\Theta^\sharp$ is a vector field dual to $\Theta$ with respect
to $g$. The $2$-form $d\Theta$ is called the {\em distance curvature
  function} and does not depend on $g\in c$. A Weyl structure whose
distance curvature function vanishes is called {\em closed}; the
cohomology class $[\Theta]\in H^1(M,\mathbf{R})$ does not depend on
$g\in c$. If the latter vanishes the Weyl structure is called {\em
  exact}. An exact Weyl structure, $\Theta=dV$, coincides with the
Levi-Civita connection of $\exp (-V)g$. More generally, a closed Weyl
structure is locally the Levi-Civita connection of a compatible
metric; it does not need to be a global metric connection unless $M$
is simply connected.

Let us fix a Riemannian metric $g\in c$ and consider the
pseudo-harmonic map equation with respect to a Weyl connection and a
metric $g$ on the domain $M$ and a metric connection on the target
$M'$,
\begin{equation}
\label{WEq}
\tau(g,\nabla^W\!\!,\nabla')(u)(x)=0,\qquad x\in M.
\end{equation}
\begin{defin*}
A map from a conformal Weyl manifold $(M,c,\nabla^W)$ to a Riemannian
manifold $(M',g')$ which solves equation~\eqref{WEq} is called {\em
  Weyl harmonic}. 
\end{defin*}

Clearly, the property of a map being Weyl harmonic does not depend on
a reference metric $g\in c$. We use below the notation $\tau^W(u)$ for
the operator $\tau(g,\nabla^W\!\!,\nabla')$. Straightforward
calculations yield 
\begin{equation}
\label{Weyl&HME}
\tau^W(u)-\tau(u)=-\left(\frac{n-2}{2}\right)du(\Theta^\sharp),
\end{equation}
where $\tau(u)$ is the harmonic map operator $\tau(g,\nabla,\nabla')$.
In particular, in dimension two equation~\eqref{WEq} coincides with
the harmonic map equation. Further, if the Higgs field $\Theta$ is
exact, $\Theta=dV$, then Weyl harmonic maps coincide with harmonic
maps with respect to the metric $\exp(-V)g$ on the domain. More
generally, for a closed Weyl structure solutions of~\eqref{WEq} are
locally harmonic maps with respect to a compatible metric on $M$ and
do not need to be global harmonic maps. This means that solutions
of~\eqref{WEq} are harmonic maps from the covering $\tilde M$ (endowed
with a compatible metric) such that the lifted form $\tilde\Theta$
becomes exact. The fundamental group $\pi_1(M)$ acts by deck
transformations on $\tilde M$ and these harmonic maps are
$\pi_1(M)$-equivariant.

In dimension greater than two there is no loss of generality in
considering Weyl harmonic map equation instead of a general
pseudo-harmonic map equation. More precisely, the following
observation holds.
\begin{prop}
Let $(M,g)$ and $(M',g')$ be Riemannian manifolds and suppose that the
former is endowed with an additional torsion-free connection
$\nabla^\dagger$. Then there exists a unique Weyl connection
$\nabla^W$ on $M$ preserving the conformal class of $g$ such that
$$
\tau(g,\nabla^W,\nabla')=\tau(g,\nabla^\dagger,\nabla').
$$
\end{prop}
\begin{proof}
The Higgs field $\Theta^\sharp$ of this Weyl connection is given by
the formula
$$
\Theta^\sharp=(2/(n-2))\mathit{trace}_g(\nabla^\dagger-\nabla),
$$
where $\nabla$ is the Levi-Civita connection of $g$.
\end{proof}

\subsection{Existence and uniqueness}
Let $(M,c,\nabla^W)$ and $(M',g')$ be a closed Weyl manifold and a
closed Riemannian manifold respectively. Suppose that the latter has
non-positive sectional curvature. Then, due to the Eells-Sampson
theorem~\cite{EeSa}, Weyl harmonic maps exist in any homotopy class
provided $\dim M=2$ or the Weyl structure is exact. For arbitrary Weyl
connections we have the following assertion.
\begin{theorem}
\label{ExTh}
Let $(M,g)$ and $(M',g')$ be closed Riemannian manifolds. Suppose
that the former has dimension greater than two and is endowed with a
Weyl connection preserving the conformal class of $g$ and the latter
has non-positive sectional curvature. Let $[v]$ be a homotopy class of
mappings $M\to M'$ such that either
\begin{itemize}
\item[(i)] it does not contain a map $v$ whose pull-back bundle
  $v^*TM'$ has a non-trivial parallel section, or
\item[(ii)] it does not contain a map onto a closed geodesic and the
  manifold $M'$ has strictly negative curvature.
\end{itemize}
Then the homotopy class $[v]$ contains a Weyl harmonic map.
\end{theorem}
The condition~$(i)$ on the homotopy class is satisfied, for
example, when the manifold $M'$ is orientable and has non-vanishing
Euler characteristic, and $v$ is non-trivial on the top cohomology
$v^*:H^{n{\scriptscriptstyle '}}(M',\mathbf{Z})\to
H^{n{\scriptscriptstyle '}}(M,\mathbf{Z})$. Indeed, under these
hypotheses the Euler class of the pull-back bundle is non-trivial and,
hence, the latter does not have any non-trivial parallel section.

When the target manifold is a locally symmetric space the hypothesis
on the Euler characteristic above is often satisfied. For example, if
the universal cover of $M'$ is a bounded symmetric domain in $\mathbf
C^m$, then $M'$ is K\"ahler hyperbolic in the sense of
Gromov~\cite{Gro}. By Gromov's solution of the Hopf-Chern conjecture
for K\"ahler hyperbolic manifolds, such a locally symmetric
space has a non-vanishing Euler characteristic (and its sign is
$(-1)^m$). Further, the Euler characteristics of orientable
irreducible locally symmetric spaces covered by
\begin{equation}
\label{UnivCover}
SO_0(p,q)/SO(p)\times SO(q)\quad\text{or}\quad Sp(p,q)/Sp(p)\times
Sp(q)
\end{equation}
also do not vanish; these cases are explained in~\cite{LR}. Since
these examples often appear below, we summarise the discussion into
the following corollary. 
\begin{corollary}
\label{Euler}
Let $(M,c,\nabla^W)$ and $(M',g')$ be a closed Weyl manifold and a
Riemannian manifold respectively. Suppose that $M'$ is a compact
quotient of a bounded symmetric domain in $\mathbf C^m$, or an
orientable irreducible quotient of one of the spaces
in~\eqref{UnivCover}. Then any map $v:M\to M'$ that is non-trivial on
the top cohomology is homotopic to a Weyl harmonic map.
\end{corollary}
\begin{remark*}
The hypotheses of Theorem~\ref{ExTh} imply that the homotopy class
$[v]$, under consideration, contains a unique harmonic
representative (with respect to any metric on the domain). For the
case when the Higgs field $\Theta$ (with respect to some $g\in c$) is
sufficiently small a Morse theory argument~\cite{GK2} yields a
stronger existence theorem: Weyl harmonic maps exist in any homotopy
class whose moduli space of harmonic mappings has a non-zero Euler
characteristic.
\end{remark*}
\begin{example}[Equivariant harmonic maps]
Let $(M^*,g^*)$ be a (not necessarily complete) Riemannian manifold of
dimension greater than two endowed with a free, co-compact, and
properly discontinuous action of a discrete group $\Gamma$ by
homotheties. Then the action of $\Gamma$ preserves the Levi-Civita
connection and the conformal class of the metric $g^*$ and these
descend to the quotient $M^*/\Gamma$ endowing the latter with the
conformal Weyl structure. Thus, the existence assertions for Weyl
harmonic maps translate into the {\em existence assertions for
  $\Gamma$-equivariant harmonic maps}; cf.~\cite{GK2}. For example, 
let $(M',g')$ be a closed manifold of negative sectional
curvature. Then Theorem~\ref{ExTh} yields that any homomorphism
$h:\Gamma\to\pi_1(M')$ whose image has a trivial centraliser is
induced by a $\Gamma$-invariant harmonic map $(M^*,g^*)\to (M',g')$. 
\end{example}

The conditions on the homotopy class in Theorem~\ref{ExTh}
can not be simply removed -- the following example shows that the
existence may fail already for closed Weyl structures.
\begin{example}[Non-existence]
\label{4:3}
Let $M'$ be a circle and $M$ be a Hopf surface, the quotient of
$\mathbf{C}^2/\{0\}$ by the action $z_i\mapsto\lambda z_i$, where
$\lambda>1$. The conformal K\"ahler metric
$$
ds^2=\frac{1}{\abs{z_1}^2+\abs{z_2}^2}\left(dz_1\otimes d\bar
z_1+dz_2\otimes d\bar z_2\right)
$$
on $\mathbf{C}^2/\{0\}$ is invariant under this action and induces a
locally conformally K\"ahler metric on $M$. Clearly, the closed
$1$-form $\Theta=-d\ln(\sum\abs{z_i}^2)$ is well-defined on $M$ and
the K\"ahler form $\Omega$ on this Hopf surface satisfies the relation
$d\Omega=\Theta\wedge\Omega$. The corresponding Weyl connection given
by~\eqref{WeylDef} is locally the Levi-Civita connection of the K\"ahler
metric $(\sum\abs{z_i}^2)ds^2$ and, hence, preserves the complex
structure on $M$. Thus, equation~\eqref{WEq} for mappings $M\to S^1$
coincides with the Hermitian harmonic map equation, Ex.~\ref{Nher}. As
is shown in~\cite[Sect.~2]{JostYau} there are no non-trivial Hermitian
harmonic maps from the Hopf surface to a circle with respect to any
Hermitian structure on the former.
\end{example}

We proceed with the discussion of the uniqueness of Weyl harmonic
maps.
\begin{theorem}
\label{UniTh}
Let $(M,g)$ and $(M',g')$ be closed Riemannian manifolds. Suppose
that the former is endowed with a Weyl connection preserving the
conformal class of $g$ and the latter has non-positive sectional
curvature. Let $u$ and $v$ be homotopic Weyl harmonic maps. Then the
maps $u$ and $v$ can be joined by a smooth one-parameter family $u_s$,
of pseudo-harmonic maps such that for each $x\in M$ the curve
$s\mapsto u_s(x)$ is a constant (independent of $x$) speed
parameterisation of a geodesic. Moreover, the correspondence $x\mapsto
(\partial/\partial s)u_s(x)$ defines a parallel section of the
pull-back bundle $u^*_sTM'$. When $M'$ has negative sectional
curvature the maps $u$ and $v$ coincide unless the rank of both $du$
and $dv$ is not greater than one everywhere.
\end{theorem}
\begin{corollary}
Under the hypotheses in Theorem~\ref{ExTh}, the homotopy class $[v]$
contains a unique Weyl harmonic representative.
\end{corollary}

The proof of Theorem~\ref{ExTh} is based on the ideas of Jost and Yau
in~\cite{JostYau}, where the authors study an analogous problem for
the Hermitian harmonic map equation. However, since no general result
on the existence of pseudo-harmonic maps is available and to make the
paper more self-contained, we give a proof in Sect.~\ref{ETproof}. The
proof of Theorem~\ref{UniTh} is similar to this for the harmonic map
equation and appears at the end of Sect.~\ref{ETproof}.

\section{First elements of the Bochner technique}
\label{firstBT}
In this section, as a warm-up, we extend Eells-Sampson
techniques~\cite{EeSa,Wu} to the setting of Weyl harmonic
maps. Apparently, these yield new topological restrictions on
manifolds with non-negative symmetric part of the Ricci-Weyl
curvature. Throughout the rest of the paper we suppose that $M$ and
$M'$ are closed manifolds and the dimension of the former is greater
than two.

Let $(M,c)$ be a closed conformal manifold and $\nabla^W$ be its Weyl
structure. Recall that due to the theorem of Gauduchon~\cite{Gaud}
there exists a unique (up to homothety) metric $g\in c$ whose Higgs
field $\Theta$, $\nabla^Wg=\Theta\otimes g$, is co-closed with respect
to $g$; i.e. its co-differential $d^*\Theta$ vanishes. By the symbol
$\Delta^W$ we denote below the Weyl laplacian on functions given by
$\mathit{trace}_g\nabla^W d$.
\begin{lemma}
\label{BLemma}
Let $(M,c,\nabla^W)$ and $(M',g')$ be a conformal Weyl manifold and an
arbitrary Riemannian manifold respectively. Suppose that $M$ is
endowed with a Gauduchon metric $g\in c$ and denote by $\Theta$ its
Higgs field. Then for any smooth solution $u$ of equation~\eqref{WEq}
the following relation holds
\begin{multline*}
\frac{1}{2}\Delta^W\norm{du}^2=\norm{\nabla du}^2-\sum_{\alpha,\beta}
\langle R'(du\cdot e_\alpha,du\cdot e_\beta)du\cdot e_\alpha,du\cdot
e_\beta\rangle_{g'}\\
+\sum_{\alpha}\left[\mathit{Ricci}^W(X_\alpha,X_\alpha)+\frac{n-2}{4}
\left(\abs{\Theta}^2\abs{X_\alpha}^2-\Theta(X_\alpha)^2\right)\right],
\end{multline*}
where $X_\alpha=(u^*\phi^\alpha)^\sharp$ and the systems
$\{e_\alpha\}$ and $\{\phi^\alpha\}$ are orthonormal bases in $T_\cdot
M$ and $T_{u{\scriptscriptstyle (\cdot)}}^*M'$ respectively at the point under
consideration; the symbol $\mathit{Ricci}^W$ denotes the Ricci
curvature of $\nabla^W$, and $R'$ stands for the curvature tensor of
the Levi-Civita connection $\nabla'$ on $M'$.
\end{lemma}
\begin{proof}
It is a simple exercise to obtain from~\eqref{WeylDef} that the Ricci
tensor of $\nabla^W$ satisfies the following relation
$$
\mathit{Ricci}^W(X,X)=\mathit{Ricci}(X,X)+\frac{n-2}{2}(\nabla_X\Theta)(X)
-\frac{n-2}{4}\left(\abs{\Theta}^2\abs{X}^2-\Theta(X)^2\right),
$$
where $\mathit{Ricci}$ denotes the Ricci tensor of a Gauduchon metric
$g$ and the symbol $\nabla$ stands for the Levi-Civita connection of
$g$. Now the lemma follows by the combination of this identity with
the Bochner formula in Lemma~\ref{Bell} (see Appendix~A).
\end{proof}

Following the lines in~\cite{Wu}, this lemma combined with the maximum
principle for $\Delta^W$ implies the following generalisation of the
well-known result in harmonic map theory.
\begin{theorem}
\label{BTh}
Let $(M,c,\nabla^W)$ and $(M',g')$ be a closed conformal Weyl manifold
and a Riemannian manifold of non-positive sectional curvature
respectively. Suppose that $M$ is endowed with a Gauduchon metric
$g\in c$ and denote by $\Theta$ its Higgs field. Suppose also that the
Ricci curvature of the Weyl connection $\nabla^W$ satisfies the
following relation
\begin{equation}
\label{RicciCond}
\mathit{Ricci}^W(X,X)+\frac{n-2}{4}\left(\abs{\Theta}^2\abs{X}^2-
\Theta(X)^2\right)\geqslant 0
\end{equation}
for any tangent vector $X\in T_{\cdot}M$. Then any Weyl harmonic map
is totally geodesic with respect to the Gauduchon metric $g$ and
$du(\Theta^\sharp)=0$. Furthermore:
\begin{itemize}
\item if at one point of $M$ the quadratic form on the left-hand side
  in~\eqref{RicciCond} has rank $k$, then the image $u(M)$ is a an
  immersed submanifold of $M'$ of dimension not greater than $(\dim
  M)-k$. In particular, if this form is positive at some point then
  $u$ is constant;
\item if $M'$ has negative sectional curvature, then $u$ is either
  constant or maps $M$ onto a closed geodesic.
\end{itemize}
\end{theorem}
Note that when the Weyl structure is metric, relation~\eqref{RicciCond}
simply means that the Ricci tensor of the Gauduchon metric is
non-negative. Certainly, condition~\eqref{RicciCond} holds when
$\mathit{Ricci}^W(X,X)\geqslant 0$; the latter condition is much
stronger and does not illustrate the complete statement. Mention, for
example, that if under the conditions of Theorem~\ref{BTh} the Ricci
curvature $\mathit{Ricci}^W$ is non-negative and $\Theta^\sharp\ne 0$
simultaniously at some point, then any Weyl harmonic map already has
to be a constant map or a map onto a closed geodesic.

Combining this theorem with the existence of Weyl harmonic maps,
guaranteed by Theorem~\ref{ExTh}, we obtain the following corollary.
\begin{corollary}
\label{12}
Let $(M,c,\nabla^W)$ be a conformal Weyl manifold whose Ricci-Weyl
tensor $\mathit{Ricci}^W$ is non-negative or satisfies the weaker
condition~\eqref{RicciCond} and $(M',g')$ be a closed non-positively
curved Riemannian manifold.
\begin{itemize}
\item[(i)] Suppose that $M'$ carries a metric of negative sectional
  curvature. Then any map $v:M\to M'$ is homotopic either to a
  constant map or a map onto a closed geodesic. In particular, any
  homomorphism from $\pi_1(M)$ to the fundamental group of a closed
  negatively curved manifold is trivial or infinite cyclic. 
\item[(ii)] Suppose $M$ and $M'$ have the same dimension and the
  latter is orientable and has non-vanishing Euler-Poincar\'e number
  $\chi(M')$. Then any map $v:M\to M'$ is trivial on the top
  cohomology $H^{n'}(M',\mathbf Z)\to H^{n'}(M,\mathbf Z)$ unless the
  Weyl structure is exact and its Ricci curvature equals zero
  identically. 
\end{itemize}
\end{corollary}
\begin{proof}
Case~$(i)$. Suppose the contrary. Let $M'$ be a closed manifold of
negative sectional curvature and $v:M\to M'$ be a smooth map which is
homotopic neither to a constant map nor a map onto a closed geodesic.
Then Theorem~\ref{ExTh} implies that $v$ is homotopic to a Weyl
harmonic map. The latter, due to Theorem~\ref{BTh}, has to be a
constant or a map onto a closed geodesic - a contradiction.

Case~$(ii)$. Suppose the contrary. Let $v:M\to M'$ be a map which is
non-trivial on the top cohomology; in particular, its degree does not
vanish. Under the hypotheses, Theorem~\ref{ExTh} applies, see the
discussion in Sect.~\ref{1:3}, -- we see that there is a Weyl harmonic
map $u$ homotopic to $v$. First, suppose that the Weyl structure
$\nabla^W$ is not exact, then the Higgs field $\Theta$ of the
Gauduchon metric can not vanish everywhere and the relation
$du(\Theta^\sharp)=0$ in Theorem~\ref{BTh} implies that
$\rank(du)<\dim M'$ in the neighbourhood of some point. Since the map
$u$ is also totally geodesic, it has constant rank and its image is a
totally geodesic submanifold whose dimension has to be less than $\dim
M'$ -- a contradiction to the fact that the degree of $u$ does not
vanish. Thus, the Higgs field $\Theta$ has to be exact and equal to
zero identically. Now Lemma~\ref{BLemma} implies that it is possible
only when the Ricci tensor $\mathit{Ricci}^W$ equals zero identically.
\end{proof}
A simple example with a torus shows that the condition on the Euler
characteristic in the last assertion is necessary. Up to our
knowledge these corollaries are first known topological obstructions
for not necessarily closed Weyl structures with non-negative symmetric
part of the Ricci-Weyl curvature in arbitrary dimension. We proceed
with an application to Einstein-Weyl geometry.

Recall that a Weyl structure is called {\em Einstein-Weyl} if the
symmetric part of the Ricci curvature $\mathit{Ricci}^W$ of the Weyl
connection is proportional to some (and hence any) metric in the
conformal class $c$. The following assertion gives topological
obstructions to the existence of non-exact Einstein-Weyl structures,
complementing the results by Gauduchon~\cite{Gaud2} and  also by
Pedersen and Swann~\cite{PS}; cf.~\cite[Th.~4.9]{CP}.
\begin{corollary}
\label{3}
Let $(M,c)$ be a conformal manifold endowed with a non-exact
Einstein-Weyl structure $\nabla^W$ and $(M',g')$ be a closed
non-positively curved Riemannian manifold. 
\begin{itemize}
\item[(i)] If $M'$ carries a metric of negative sectional curvature,
  then any map $v:M\to M'$ is homotopic either to a constant map or a
  map onto a closed geodesic. In particular, any homomorphism from
  $\pi_1(M)$ to the fundamental group of a closed negatively curved
  manifold is trivial or infinite cyclic. 
\item[(ii)] If $M'$ is orientable and its Euler characterictic of
  $\chi(M')$ does not vanish, then any map $v:M\to M'$ is trivial
  on the top cohomology $H^{n'}(M',\mathbf Z)\to H^{n'}(M,\mathbf
  Z)$. In particular, the manifolds $M$ and $M'$ are not homotopy
  equivalent.
\end{itemize}
\end{corollary}
\begin{proof}
First, due to the important result of Gauduchon~\cite[Th.3]{Gaud2}
the Ricci curvature $\mathit{Ricci^W}$ of a compact Einstein-Weyl
manifold with a closed Weyl structure vanishes identically unless the
latter is exact. The combination of this with Theorems~\ref{ExTh}
and~\ref{BTh} proves the statements when the Einstein-Weyl structure
is closed. More precisely, the first statement is a consequence of
Cor.~\ref{12}. The proof of the second is based on the observation
that, since $\mathit{Ricci}^W$ vanishes identically and $\Theta$ does
not, any Weyl harmonic map $u$ is constant or maps $M$ onto a closed
geodesic. Thus, the pull-back bundle $u^*TM'$ has trivial Euler class
and, since $\chi(M')\ne 0$, the map $u$ has to be trivial on the top
cohomology.

Now the classification theorem of Einstein-Weyl
structures~\cite[Th.~4.7]{CP} says that non-closed Einstein-Weyl
structures occur only when $M$ admits a metric of positive Ricci
curvature or the dimension of $M$ is at most three. The first
possibility is also handled by Theorems~\ref{ExTh} and~\ref{BTh}.
In dimension three Einstein-Weyl structures on compact manifolds are
completely classified by Tod~\cite{Tod}. In particular, non-closed
ones occur only on manifolds $M$ which are finitely covered by the
sphere $S^3$. In this case the conclusions of the statement hold
simply by topological reasons. In more detail, since $\pi_1(M')$ does
not have non-trivial elements of finite order, any map $v:M\to M'$
induces a trivial map on the fundamental groups. Now since $M'$ is a
$K(\pi,1)$-space, the map $v$ has to be null-homotopic.
\end{proof}
We end with two remarks. Firstly, Cor.~\ref{3} implies rigidity of
Einstein-Weyl structure on some irreducible locally symmetric spaces
of non-compact type: any Einstein-Weyl structure $(c,\nabla^W)$
has to be genuine Einstein. Such locally symmetric spaces include
orientable spaces whose universal cover is a bounded symmetric domain
in $\mathbf C^m$ or one of the spaces
$$
SO_0(p,q)/SO(p)\times SO(q),\quad Sp(p,q)/Sp(p)\times Sp(q),\text{ or }
F_{-4(20)}/Spin(9);
$$
see Cor.~\ref{Euler}. This is in contrast with the symmetric spaces of
compact type, where non-trivial Einstein-Weyl structures do exist;
see~\cite{Kerr}.

Secondly, mention that the corollaries above also imply the
non-existence of non-trivial Einstein-Weyl structures (or structures
with non-negative symmetric part of Ricci-Weyl curvature) on certain
connected sums. More precisely, let $M'$ be a manifold as in
Cor.~\ref{3} which satisfies $(i)$ or $(ii)$ and $M$ be an
arbitrary closed manifold of the same dimension. Then their connected
sum does not admit a non-trivial Einstein-Weyl structure
$(c,\nabla^W)$. Indeed, otherwise we  would have a contradiction with
Cor.~\ref{3}, since there is a map of degree one from $M'\sharp
M$ to $M'$ (for example, obtained by collapsing $M$ into a point).

\section{Pluriharmonicity of Weyl harmonic maps}
\label{PLURI}

\subsection{Preliminaries on K\"ahler-Weyl structures}
Let $M$ be a complex manifold of complex dimension $n>1$ and
let $J$ be its complex structure. A conformal manifold $(M,c,J)$ is
called {\it Hermitian} if $c$ contains a Hermitian metric; thus, any
metric in $c$ is Hermitian. Following~\cite{CP}, the triple
$(M,c,\nabla^W)$ is called a {\it K\"ahler-Weyl manifold} if the Weyl
connection $\nabla^W$ preserves the compex structure on $M$. 

Straightforward calculation shows that in complex dimension two any
Hermitian conformal manifold $(M,c,J)$ admits a unique K\"ahler-Weyl
structure; i.e. a unique Weyl connection $\nabla^W$ such that
$\nabla^WJ=0$. In higher dimensions the existence of such a Weyl
connection forces $M$ to be a {\it locally conformal K\"ahler}
manifold. We recall the definition of this notion now.

A Hermitian manifold $(M,g,J)$ is called locally conformal K\"ahler
if there exists an open covering $\{U_i\}$ of $M$ such that any
restriction $\left.g\right|_{U_i}$ is conformally equivalent to a
K\"ahler metric $\tilde g_i$,
$$
g(x)=\exp(V_i(x))\tilde g_i(x), \quad\text{where }x\in U_i.
$$
This is equivalent to the condition:
$$
d\omega=\Theta\wedge\omega,\qquad d\Theta=0,
$$
where $\omega$ is the K\"ahler form of $g$ and $\Theta$ is a closed
$1$-form (called the {\it Lee form}), locally defined by
$\left.\Theta\right|_{U_i}=dV_i$. In particular, the form $\Theta$ is
exact if and only if $(M,g,J)$ is {\it globally conformal
  K\"ahler}. The Levi-Civita connections of the local K\"ahler metrics
$\tilde g_i$ glue together to a connection $\nabla^W$ related to the
Levi-Civita connection $\nabla$ of the metric $g$ by the formula
$$
\nabla^W_XY=\nabla_XY-\frac{1}{2}\Theta(X)Y-\frac{1}{2}\Theta(Y)X+
\frac{1}{2}g(X,Y)\Theta^\sharp.
$$
The two properties $\nabla^Wg=\Theta\otimes g$ and $\nabla^W$ is
torsion-free show that $\nabla^W$ is a Weyl connection on the
conformal Hermitian manifold $(M,[g],J)$. Besides, the connection
$\nabla^W$ satisfies $\nabla^WJ=0$ and its Higgs field is precisely
the Lee form $\Theta$.

Recall that due to the theorem of Gauduchon~\cite{Gaud} there exists a
canonical Hermitian metric $g\in c$ such that the corresponding Lee
form $\Theta$ is co-closed,
$$
-g^{\alpha\bar\mu}(\nabla\Theta)_{\alpha\bar\mu}=0;
$$
here $\nabla$ denotes the Levi-Civita connection of $g$.
\begin{defin*}
A metric $g\in c$ on a K\"ahler-Weyl manifold $(M,c,\nabla^W)$ is
called {\it pluricanonical}, if the $(1,1)$-part of the covariant
derivative of the corresponding Lee form vanishes, 
$(\nabla\Theta)_{\alpha\bar\mu}=0$.
\end{defin*}
First, any pluricanonical metric is canonical in the sense of the
Gauduchon theorem and, in particular, is unique up to
homothety. Second, if $M$ admits a Kahler metric, then
the latter is obviously pluricanonical. The large and well-studied
class of locally conformal K\"ahler manifolds which admit
pluricanonical metrics is formed by {\it generalised Hopf manifolds}
(or also called {\it Vaisman manifolds}), see~\cite{DO,Ornea}. These
are manifolds satisfying a much stronger hypothesis: they admit
metrics with the parallel Lee form, $\nabla\Theta=0$.

Generalised Hopf manifolds with non-trivial Lee form do not admit
K\"ahler metrics. The reason for this is topological -- their first
Betti number is odd, see~\cite{DO,Ornea}. More generally, the
{\it Vaisman conjecture} states that any locally conformal K\"ahler
manifold of the ``same topology'' as a K\"ahler manifold admits a
global K\"ahler metric. For complex surfaces the situation is
satisfactory: as is known any complex surface with even first Betti
number is K\"ahler, see~\cite[Sect.~1]{Am}. However, in higher
dimensions there are examples of locally conformal K\"ahler manifolds
with even first Betti numbers which do not admit K\"ahler metrics,
see~\cite{OT}.

In Sect.~\ref{top} we give a partial confirmation of Vaisman's
principle: we show that locally conformal K\"ahler manifolds with
pluricanonical metrics which are homotopy equivalent to compact
quotients of bounded symmetric domains in $\mathbf C^n$ without a unit
disk in $\mathbf C$ as a factor have to be K\"ahler.

\subsection{Bochner-Sampson technique}
Now we discuss a Bochner-type formula for Weyl harmonic maps from a
K\"ahler-Weyl manifold $M$ to a Riemannian manifold $M'$. This formula
is obtained by computing an iterated Weyl-divergence of the symmetric
$(1,1)$-tensor 
$$
\psi(X,Y)=e(u)\langle X,Y\rangle-\langle d'u(X),d^{\prime\prime}
u(Y)\rangle,
$$
where $u$ is a Weyl harmonic map and $e(u)$ is its energy density
$(1/2)\lVert du\rVert^2$. The first and the second inner products
above are the complex linear extensions of the Riemannian metrics on
$M$ and $M'$ respectively to the complexified tangent bundles; the
$d'u$  and $d''u$ are the restriction of the complexification of $du$
to the holomorphic and anti-holomorphic tangent bundles respectively.
Since $M'$ is not assumed to have a complex structure, the
differentials $d'u$ and $d''u$ are not the ones usual to complex
manifold theory.

Consider next the Weyl harmonic map equation. If $z^\alpha$ denote
local holomorphic coordinates on $M$ and the $u^i$ denote local smooth
coordinates on $M'$, then $d'u$ is represented by the matrix
$(u^i_\alpha)$, where subscripts denote differentiation with respect
to $z^\alpha$. The Weyl harmonic map equation takes the form
\begin{equation}
\label{CWEq}
g^{\alpha\bar\mu}\nabla^W_{\bar\mu}u^i_{\alpha}=0,
\end{equation}
where $g$ is an arbitrary (Hermitian) metric in $c$. Here the
$\nabla^W_{\bar\mu}u^i_{\alpha}$ are the components of the tensor
$\tilde\nabla''d'u$, defined as the $(0,1)$-part of the covariant
differential of $d'u$ in the natural connection on $\Hom(T^{1,0}M,
u^*T^cM')$, i.e. that determined by the Weyl connection on $TM$
and the Levi-Civita connection on $TM'$. Since the Weyl connection
$\nabla^W$ is complex, equation~\eqref{CWEq} coincides with the
Hermitian harmonic map equation; see Ex.~\ref{Nher}. 

Now construct a $(1,0)$-form by the divergence-type formula
$$
\xi_\alpha=g^{\beta\bar\gamma}\nabla^W_{\beta}
\psi_{\alpha\bar\gamma}.
$$
The latter satisfies the following extension of Sampson's 
formula~\cite{Sam} to the context of pseudo-harmonic maps.
\begin{lemma}
\label{SamF}
Let $(M,c,\nabla^W)$ be a K\"ahler-Weyl manifold and $(M',g')$ be a
Riemannian manifold. Then for any metric $g\in c$ a Weyl harmonic map
$u:M\to M'$ satisfies the following relation
\begin{equation}
\label{BSf}
g^{\alpha\bar\mu}\nabla^W_{\bar\mu}\xi_\alpha+
\xi_\alpha\Theta^\alpha=g'_{ij}(\nabla^W_{\bar\gamma}u^i_\alpha)
(\nabla^W_{\bar\mu}u^j_\beta)g^{\alpha\bar\mu}g^{\beta\bar\gamma}
-R'_{ijlm}u^i_\alpha u^j_\beta u^l_{\bar\mu}u^m_{\bar\gamma}
g^{\alpha\bar\mu}g^{\beta\bar\gamma},
\end{equation}
where $\Theta^\alpha$ are $(1,0)$-components of the Lee field
$\Theta^\sharp$.
\end{lemma}
\begin{proof}
In local coordinates the quadratic differential
$\psi_{\alpha\bar\gamma}$ has the form
$$
\psi_{\alpha\beta}=e(u)g_{\alpha\bar\gamma}-g'_{ij}u^i_\alpha
u^j_{\bar\gamma},
$$
where $e(u)$ equals $g^{\delta\bar\nu}u^k_\delta
u^l_{\bar\nu}g'_{kl}$. Its Weyl-divergence, the form $\xi_\alpha$,
satisfies the following relation
\begin{equation}
\label{R1}
\xi_\alpha=g'_{ij}(\nabla^W_{\bar\gamma}u^i_\alpha)u^j_\beta 
g^{\beta\bar\gamma}.
\end{equation}
Indeed, by straightforward differentiation, we have
\begin{multline*}
\xi_\alpha=g^{\beta\bar\gamma}\nabla^W_\beta\psi_{\alpha\bar\gamma}=
g^{\delta\bar\nu}(\nabla^W_\alpha u^k_\delta)u^l_{\bar\nu}g'_{kl}+
g^{\delta\bar\nu}u^k_\delta(\nabla^W_\alpha u^l_{\bar\nu})g'_{kl}\\
-g^{\beta\bar\gamma}g'_{ij}(\nabla^W_\beta u^i_\alpha)u^j_{\bar\gamma}
-g^{\beta\bar\gamma}g'_{ij}u^i_\alpha(\nabla^W_\beta u^j_{\bar\gamma}).
\end{multline*}
Since the map $u$ satisfies equation~\eqref{CWEq}, the last term above
vanishes. Further, since the form $\nabla^Wdu$ is symmetric, the first
and the third terms cancel out and the second coincides with the
right-hand side in~\eqref{R1}.

Now we compute the Weyl-divergence of $\xi_\alpha$. Using
formula~\eqref{R1} and the fact that the derivative $\nabla^W_{\bar\mu}
g^{\beta\bar\gamma}$ equals $-\Theta_{\bar\mu}g^{\beta\bar\gamma}$, we
obtain
$$
g^{\alpha\bar\mu}\nabla^W_{\bar\mu}\xi_\alpha=g^{\alpha\bar\mu}
\nabla^W_{\bar\mu}\left(\nabla^W_{\bar\gamma}u^i_\alpha
u^j_\beta\right)g^{\beta\bar\gamma}g'_{ij}-g^{\alpha\bar\mu}
\Theta_{\bar\mu}\left(\nabla^W_{\bar\gamma}u^i_\alpha\right)
u^j_\beta g^{\beta\bar\gamma}g'_{ij}.
$$
Clearly, the second term above equals
$\xi_\alpha\Theta^\alpha$. Following the lines in~\cite{Sam}, the
first term in the right-hand side above can be further transformed
as follows
\begin{equation}
\label{R4}
g'_{ij}(\nabla^W_{\bar\gamma}u^i_\alpha)(\nabla^W_{\bar\mu}u^j_\beta)
g^{\alpha\bar\mu}g^{\beta\bar\gamma}-R'_{ijlm}u^i_\alpha u^j_\beta 
u^l_{\bar\mu}u^m_{\bar\gamma}g^{\alpha\bar\mu}g^{\beta\bar\gamma}
+g'_{ij}\left(\nabla^W_{\bar\gamma}\nabla^W_{\bar\mu}u^i_\alpha\right)
u^j_\beta g^{\alpha\bar\mu}g^{\beta\bar\gamma}.
\end{equation}
Indeed, it equals the sum
$$
g'_{ij}\left(\nabla^W_{\bar\mu}\nabla^W_{\bar\gamma}u^i_\alpha\right)
u^j_\beta g^{\alpha\bar\mu}g^{\beta\bar\gamma}+
g'_{ij}(\nabla^W_{\bar\gamma}u^i_\alpha)(\nabla^W_{\bar\mu}u^j_\beta)
g^{\alpha\bar\mu}g^{\beta\bar\gamma}.
$$
Due to the fact that the Weyl connection $\nabla^W$ preserves the
complex structure $J$, we have
\begin{equation}
\label{MainF}
\nabla^W_{\bar\mu}\nabla^W_{\bar\gamma}u^i_\alpha-\nabla^W_{\bar\gamma}
\nabla^W_{\bar\mu}u^i_\alpha=-R^{\prime~i}_{jlm}u^j_\alpha
u^l_{\bar\mu} u^m_{\bar\gamma}.
\end{equation}
Substituting this into the first term above we arrive at the
expression~\eqref{R4}. To prove formula~\eqref{BSf} it remains to
show that the last term in~\eqref{R4} vanishes. In fact, a stronger
relation holds: the tensor $g^{\alpha\bar\mu}\nabla^W_{\bar\gamma}
\nabla^W_{\bar\mu}u^i_\alpha$ vanishes. To see the latter we
differentiate the Weyl harmonic map equation~\eqref{CWEq}, 
$$
0=\nabla^W_{\bar\gamma}\left(g^{\alpha\bar\mu}\nabla^W_{\bar\mu}
u^i_\alpha\right)=g^{\alpha\bar\mu}\nabla^W_{\bar\gamma}
\nabla^W_{\bar\mu}u^i_\alpha-\Theta_{\bar\gamma}g^{\alpha\bar\mu}
\nabla^W_{\bar\mu}u^i_\alpha.
$$ 
Since $u$ is a Weyl harmonic map, the last term above vanishes, and
hence so does the first term on the right-hand side. 
\end{proof}

Recall that a Riemannian manifold $M'$ is said to have {\it
  non-positive Hermitian sectional curvature} if the extension of the
  curvature tensor $R'$ to the complexified tangent bundle $T^cM'$
  satisfies the relation
\begin{equation}
\label{HSCurv}
\langle R'(X,Y)\bar X,\bar Y\rangle\leqslant 0\quad
\text{for all}\quad X, Y\in T^cM'.
\end{equation}
If the vector fields $X$ and $Y$ are real and orthonormal, then the
curvature quantity above is precisely the sectional curvature of the
plane they determine. However, in general the condition~\eqref{HSCurv}
is stronger than the condition of non-positive sectional curvature.
\begin{theorem}
\label{pluri}
Let $(M,c,\nabla^W)$ be a K\"ahler-Weyl manifold and $(M',g')$ be a
Riemannian manifold of non-positive Hermitian sectional curvature. Let
$u:M\to M'$ be a Weyl harmonic map. If $M$ has complex dimension $2$
or admits a pluricanonical metric, then $u$ is pluriharmonic,
$\tilde\nabla''d'u=0$, and
$$
\langle R'(X,Y)\bar X,\bar Y\rangle=0\quad
\text{for all}\quad X, Y\in du(T_x^{1,0}M),\quad x\in M.
$$
\end{theorem}
\begin{remark*}
The vanishing of the curvature term is, in fact, a consequence of the
pluriharmonicity; see relation~\eqref{MainF}.
\end{remark*}
\begin{proof}
We prove the theorem by integrating the formula in Lemma~\ref{SamF}. 
More precisely, the left-hand side of formula~\eqref{BSf} equals
$$
-d^*\xi-(n-2)\xi_\alpha\Theta^\alpha,
$$
and by the curvature hypothesis each term in the right-hand side
is non-negative. In the case $n=2$, this directly implies the
statement. For the rest of the proof we suppose that $n\geqslant
3$ and, hence, by the curvature hypothesis, the integral
$$
\int_M(\xi_\alpha\Theta^\alpha)\mathit{dVol}
$$
is non-positive. We claim that for a pluricanonical metric $g\in c$
this integral is also non-negative. Indeed, since the form
$\xi_\alpha$ equals $g^{\beta\bar\gamma}\nabla^W_\beta
\psi_{\alpha\bar\gamma}$, the integration by parts yields
\begin{equation}
\label{II}
\int_M(\xi_\alpha\Theta^\alpha)\mathit{dVol}=
-(n-2)\int_M(\psi_{\alpha\bar\gamma}\Theta^\alpha\Theta^{\bar\gamma})
\mathit{dVol}-\int_M\psi_{\alpha\bar\gamma}(g^{\beta\bar\gamma}
g^{\alpha\bar\mu}\nabla^W_\beta\Theta_{\bar\mu})\mathit{Vol}.
\end{equation}
The relation between the Weyl and Levi-Civita
connections also yields the identity
$$
(\nabla^W\!\!\Theta)_{\beta\bar\mu}=(\nabla\Theta)_{\beta\bar\mu}+
\Theta_\beta\Theta_{\bar\mu}-(1/2)g_{\beta\bar\mu}\abs{\Theta}^2.
$$
Since the metric $g$ is pluricanonical, the term
$(\nabla\Theta)_{\beta\bar\mu}$ vanishes, and substituting the above
into formula~\eqref{II}, we obtain
$$
\int_M(\xi_\alpha\Theta^\alpha)\mathit{dVol}=
-(n-1)\int_M(\psi_{\alpha\bar\gamma}\Theta^\alpha\Theta^{\bar\gamma})
\mathit{Vol}+(1/2)\int_M(g^{\alpha\bar\gamma}\psi_{\alpha\bar\gamma})
\abs{\Theta}^2\mathit{dVol}.
$$
It is straightforward to see that the first integrand on the
right-hand side satisfies the relation
$$
\psi_{\alpha\bar\gamma}\Theta^\alpha\Theta^{\bar\gamma}=
(1/2)e(u)\abs{\Theta}^2-g'_{ij}u^i_\alpha u^j_{\bar\gamma}
\Theta^\alpha\Theta^{\bar\gamma}
$$
and the second equals $(n-1)e(u)\abs{\Theta}^2$. Thus, we arrive at
the formula
$$
\int_M(\xi_\alpha\Theta^\alpha)\mathit{dVol}=
(n-1)\int_M(g'_{ij}u^i_\alpha u^j_{\bar\gamma}\Theta^\alpha
\Theta^{\bar\gamma})\mathit{dVol}.
$$
Since the right-hand side here is real and non-negative, we are done.
\end{proof}

\subsection{Abelian subalgebras of symmetric spaces}
We specialise the considerations above to the case when $M'$ is a
locally symmetric space of non-compact type. This means that the
universal cover of $M'$ is a symmetric space $G/K$, where $G$ is a
connected Lie group whose Lie algebra $\mathfrak g$ is semi-simple and
non-compact, $K$ is a maximal compact subgroup of $G$, and $G/K$ is
given the invariant metric determined by the Killing form. We have the
Cartan decomposition $\mathfrak g=\mathfrak k\oplus\mathfrak p$,
where $\mathfrak k$ is the Lie algebra of $K$ and $\mathfrak p$ is a
$K$-invariant complement of $\mathfrak k$. The Killing form is
positive definite on $\mathfrak p$ and negative definite on $\mathfrak
k$; see~\cite{Hel}. We also have the relations
$$
[\mathfrak k,\mathfrak p]\subset\mathfrak p,\qquad
[\mathfrak p,\mathfrak p]\subset\mathfrak k.
$$
For our purposes it will be harmless to identify the tangent space
$T_uM'$ at any fixed point $u\in M'$ with $\mathfrak p$. Under this
identification the curvature tensor is given by
$$
R'(X,Y)Z=[[X,Y],Z],\qquad X,Y,Z\in\mathfrak p,
$$
and the Hermitian sectional curvature on $T^cM'$ is given by
$$
\langle R'(X,Y)\bar X,\bar Y\rangle=\langle[X,Y],[\bar X,\bar
  Y]\rangle,\qquad X,Y\in\mathfrak p^c.
$$
The latter is non-positive, and vanishes if and only if $[X,Y]=0$,
since the Killing form is negative definite on $\mathfrak k$.

Now we can re-phrase Theorem~\ref{pluri} in the following form:
\begin{theorem}
\label{AP}
Let $(M,c,\nabla^W)$ be a K\"ahler-Weyl manifold and $(M',g')$ be a
locally symmetric space of non-compact type. Let $u:M\to M'$ be a
Weyl harmonic map. If $M$ has complex dimension $2$ or admits a
pluricanonical metric, then $u$ is pluriharmonic and for any $x\in M$
the image of $T^{1,0}_xM$ under $du(x)$ is an Abelian subspace of
$\mathfrak p^c$.
\end{theorem}

Denote by $\mathfrak a$ the image of $T^{1,0}_xM$ under $du(x)$. The
corresponding image of the real tangent space $T_xM$ is the subspace
of real vectors of the space $\mathfrak a+\bar{\mathfrak a}$, so that
$$
\dim_{\mathbf R}du(T_xM)=\dim_{\mathbf C}(\mathfrak a+\bar{\mathfrak
  a})\leqslant 2\dim_{\mathbf C}\mathfrak a.
$$
Combining this inequality with Theorem~\ref{AP}, we obtain the
following estimate:
\begin{equation}
\label{RE}
\rank(du)\leqslant 2\max\{\dim(\mathfrak a): \mathfrak a\subset\mathfrak
p^c, [\mathfrak a,\mathfrak a]=0\}.
\end{equation}
Observe that if $G/K$ is a Hermitian symmetric space, then
corresponding to an invariant complex structure on $G/K$ we have the
decomposition 
$$
\mathfrak p^c=\mathfrak p^{1,0}\oplus\mathfrak p^{0,1}.
$$
Recall that the complex structure on $G/K$ belongs to the centre of
the linear isotropy algebra~\cite[VIII.4.5]{Hel} and, hence, extends
to $\mathfrak g^c$ as a derivation which vanishes on $\mathfrak
k^c$. This together with the relation $[\mathfrak p,\mathfrak p]
\subset\mathfrak k$ implies that $\mathfrak p^{1,0}$ is an Abelian
subspace of $\mathfrak p^c$.\footnote{As was pointed out by
  D.~Alekseevsky, the argument in~\cite{CT89} used to show that
  $\mathfrak p^{1,0}$ is an Abelian subspace is incorrect.}

The following algebraic theorem is due to Carlson and
Toledo~\cite{CT89}.
\begin{abel1}
Let $G/K$ be a symmetric space of non-compact type that does not
contain the hyperbolic plane as a factor. Let $\mathfrak
a\subset\mathfrak p^c$ be an Abelian subspace. Then $\dim(\mathfrak a)
\leqslant (1/2)\dim(\mathfrak p^c)$. Besides, the equality holds if
and only if $G/K$ is Hermitian symmetric and $\mathfrak a=\mathfrak
p^{1,0}$ for an invariant complex structure on $G/K$.
\end{abel1}

As an immediate consequence we have the following statement.
\begin{corollary}
\label{rank1}
Let $M$ and $M'$ be as in Theorem~\ref{AP} and $u:M\to M'$ be a
Weyl-harmonic map. Suppose also that $M'$ is not locally Hermitian
symmetric. If $M$ has complex dimension $2$ or admits a pluricanonical
metric, then for every $x\in M$, the rank $du(x)$ is strictly
smaller than the dimension of $M'$.
\end{corollary}

Above we used only the case of strict inequality in the Abelian
subspaces theorem. The case of equality yields the following extension
of Siu's result to Weyl-harmonic maps.
\begin{theorem}
\label{Siu1}
Let $(M,c,\nabla^W)$ be a K\"ahler-Weyl manifold and $(M',g')$
be a locally Hermitian symmetric space of non-compact type whose
universal cover does not contain the hyperbolic plane as a factor. Let
$u:M\to M'$ be a Weyl harmonic map such that the rank of $du(x)$
equals the dimension of $M'$ for some $x\in M$. If $M$ has complex
dimension $2$ or admits a pluricanonical metric, then $u$ is
holomorphic with respect to an invariant complex structure on $M'$.
\end{theorem}
\begin{proof}
Since $du(T^{1,0}_xM)$ is an Abelian subspace of the half dimension,
it must be $\mathfrak p^{1,0}$ for an invariant complex structure on
$M'$. Thus $du(x)$ maps $T^{1,0}_xM$ into $\mathfrak p^{1,0}$ and,
hence, is complex linear. Since having a maximal rank is an open
condition, the linear operator $du(x)$ is complex linear on an open
subset of $M$; i.e. the map $u$ is holomorphic on an open subset of
$M$. Now the version of Siu's unique continuation theorem for
Hermitian harmonic maps (see Appendix~B) implies that $u$ is
holomorphic everywhere with respect to an invariant complex structure
on $M'$.
\end{proof}

For each symmetric space $G/K$ let $\nu(G/K)$ denote the maximum
complex dimension of an Abelian subspace of $\mathfrak p^c$. By
Abelian subspaces theorem~I, we have the inequality $\nu(G/K)<(1/2)
\dim(G/K)$ provided $G/K$ is not Hermitian symmetric. For hyperbolic
spaces over various division algebras the following theorem, due
to~\cite{CT89} and~\cite{CH}, gives more precise information.
\begin{abel2}
Let $H^m_{\mathbf K}$ denote the hyperbolic space of $\mathbf
K$-dimension $m$ over the division algebra $\mathbf K$, where $\mathbf
K=\mathbf R$, or $\mathbf C$, or the quaternions $\mathbf H$, or the
octonions $\mathbf O$. Let $\mathfrak a$ be an Abelian subspace of
$\mathfrak p^c$ where $\mathfrak g=\mathfrak k\oplus\mathfrak p$ is
the Cartan decomposition of the group of isometries of $H^m_{\mathbf
  K}$. Then:
\begin{itemize}
\item[(i)] if $\mathbf K=\mathbf R$, $\nu(H^m_{\mathbf K})=1$.
\item[(ii)] if $\mathbf K=\mathbf C$, and $\dim(\mathfrak a)>1$, then
  $\mathfrak a\subset\mathfrak p^{1,0}$ for one of the two invariant
  complex structures on $H^m_{\mathbf C}$; besides, we have
  $\nu(H^m_{\mathbf K})=m$.
\item[(iii)] if $\mathbf K=\mathbf H$, $\nu(H^m_{\mathbf K})=m$.
\item[(iv)] if $\mathbf K=\mathbf O$, $\nu(H^2_{\mathbf K})=2$.
\end{itemize}
\end{abel2}
\begin{corollary}
\label{rank2}
Let $M$ be a compact K\"ahler-Weyl manifold which has complex
dimension $2$ or admits a pluricanonical metric. Let $u:M\to M'$ be a
Weyl-harmonic map, where $M'$ is a quotient of $H^m_{\mathbf K}$.
Then:
\begin{itemize}
\item[(i)] if $\mathbf K=\mathbf R$, the rank of $u$ is at most two.
\item[(ii)] if $\mathbf K=\mathbf C$ and the rank of $u$ exceeds two
  at some point $x\in M$, then $u$ is holomorphic with respect to one
  of the two complex structures on $M'$.
\item[(iii)] if $\mathbf K=\mathbf H$, the rank of $u$ is at most $2m$,
  one-half of the (real) dimension of $M'$.
\item[(iv)] if $\mathbf K=\mathbf O$, the rank of $u$ is at most four.
\end{itemize}
\end{corollary}
\begin{proof}
The cases $(i)$, $(iii)$, and $(iv)$ are an immediate application of
the rank estimate~\eqref{RE}. The case~$(ii)$ follows in the same
fashion as Theorem~\ref{Siu1} follows from the Abelian subspaces
theorem~I.
\end{proof}

\section{Applications to conformal K\"ahler geometry}
\label{top}
\subsection{Topological comparisons of K\"ahler-Weyl manifolds
  and locally symmetric spaces.}
In this section we collect applications of the results above to the
topology of K\"ahler-Weyl manifolds. Our conclusions are extensions of
the results for K\"ahler manifolds to a significantly larger class of
complex manifolds. We start with the following rigidity theorem.
\begin{theorem}
\label{Siu2}
Let $(M,c,\nabla^W)$ be a K\"ahler-Weyl manifold and $(M',g')$ be a
locally Hermitian symmetric space of non-compact type whose universal
cover does not contain the hyperbolic plane as a factor. Suppose that
there exists a homotopy equivalence $u:M\to M'$. If $M$ has complex
dimension $2$ or admits a pluricanonical metric, then $u$ is homotopic
to a biholomorphism for some invariant complex structure on $M'$.
\end{theorem}
\begin{proof}
Since the homotopy equivalence $u$ induces the isomorphism on the
cohomology, Cor.~\ref{Euler} implies that $u$ is homotopic to a
Weyl-harmonic map, which we also denote by $u$. Since $u$ can not have
the rank less than the dimension of $M'$ everywhere, by
Theorem~\ref{Siu1} it is holomorphic. Now we claim that the existence
of the holomorphic homotopy equivalence $u:M\to M'$ implies that $M$
admits a global K\"ahler metric.

First, suppose that $M$ is a complex surface. Then, since it is
homotopy equivalent to a Kahler manifold, its first Betti number is
even and it has to be itself K\"ahler, see~\cite[Sect.~1]{Am}. Now
suppose that the complex dimension $n>2$ and, hence, $M$ is locally
conformal K\"ahler. Choose a background metric $g\in c$ and let
$\Theta$ be its Lee form. It is straightforward to show that the {\it
  anti-Lee} form $\Theta^\dagger=\Theta\circ J$ satisfies the relation
$d\Theta^\dagger=2i\bar\partial \Theta^{1,0}$, i.e. $d\Theta^\dagger$
is a real exact form of type $(1,1)$. Since $u$ is holomorphic and is
an isomorphism on the first cohomology, replacing the metric on $M$ by
the conformal one we can suppose that $\Theta$ is the pull-back of a
closed $1$-form $\Lambda$ on $M'$ and $\Theta^\dagger$ is the
pull-back of $\Lambda^\dagger$, where $\Lambda^\dagger=\Lambda\circ
J$. But on a K\"ahler manifold the $(1,1)$-form $d\Lambda^\dagger$ is
$\partial\bar\partial$-exact and, hence, so is the form
$d\Theta^\dagger$. In other words, there exists a function $\varphi$
on $M$ such that $d\Theta^\dagger=2i\partial\bar\partial\varphi$.
Now the computation of Vaisman in~\cite{Va1} shows that the metric $h=
\exp(\varphi)g$ is in fact K\"ahler. We describe it below for the
completeness.

Let $\tilde\Theta$ be the Lee form of the new metric $h$. The above
yields $\bar\partial\Theta^{1,0}=\partial\bar\partial\varphi$ and,
hence, the $(1,0)$-part of $\tilde\Theta$ is a holomorphic form,
$$
\bar\partial\tilde\Theta^{1,0}=\bar\partial\Theta^{1,0}+\bar\partial
\partial\varphi=0.
$$
Since the Weyl connection is complex, this means that
$\nabla^W_{\bar\mu}\tilde\Theta_\alpha=0$ and in terms of the
Levi-Civita connection of $h$ reads as
$$
-\tilde\Theta_{\bar\mu}\tilde\Theta_\alpha+h_{\alpha\bar\mu}
h^{\beta\bar\gamma}\tilde\Theta_\alpha\tilde\Theta_{\bar\gamma}=0.
$$
The latter yields the relation
$$
d^*_h\tilde\Theta^{1,0}=
-((n-1)/2)\abs{\tilde\Theta}_h^2.
$$
Now the integration over $M$ implies that $\tilde\Theta$ vanishes,
demonstrating the claim.

Finally, we show that $u$ is a biholomorphism. Suppose the
contrary. Then its singular locus $C$, defined by $\det(\partial
u/\partial z)$=0, is a complex-analytic subvariety of $M$ and defines
a non-trivial homology class in $H_{2n-2}(M,\mathbf R)$. By the
argument of Siu~\cite{Siu}, the image $u(C)$ is a subvariety of $M'$
of complex codimension at least two. More precisely, for otherwise
there exists a point $p\in C$ that is isolated in
$u^{-1}(u(p))$. Using a local coordinate chart $(z^\alpha)$ on $M$ at
$p$, we see that the Riemann removable singularity theorem applies to
$z^\alpha\circ u^{-1}$ on $V\backslash u(C)$ for some open
neighbourhood $V$ of $u(p)$ in $M'$ and $u$ is locally diffeomorphic
at $p$. This contradicts to the supposition that $p\in C$. Thus the
image $u(C)$ is null-homologous in $H_{2n-2}(M',\mathbf R)$,
contradicting to the fact that $u$ is a homotopy equivalence.
\end{proof}

The rank estimates in Sect.~\ref{PLURI} imply the following
topological restrictions on maps from compact K\"ahler-Weyl manifolds
to locally symmetric spaces of non-Hermitian type.
\begin{corollary}
\label{top1}
Let $(M,c,\nabla^W)$ be a K\"ahler-Weyl manifold and $(M',g')$ be an
orientable compact locally symmetric space of non-compact
type. Suppose that $M'$ is not locally Hermitian symmetric and its
Euler characteristic does not vanish (for example, one of the spaces in
Cor.~\ref{Euler}). If $M$ has complex dimension $2$ or admits a
pluricanonical metric, then any map $v:M\to M'$ is trivial on the top
cohomology $H^{n'}(M',\mathbf Z)\to H^{n'}(M,\mathbf Z)$. 
\end{corollary}
\begin{proof}
Suppose the contrary: there exists a map $v:M\to M'$ non-trivial on
the top cohomology. Then by Theorem~\ref{ExTh} $v$ is homotopic to a
Weyl harmonic map $u$. Now Cor.~\ref{rank1} implies that $u$ is not
surjective, hence by standard topology, can be deformed to a map
whose image lies in a proper subskeleton of some cell subdivision of
$M'$. In particular, the latter vanishes on the top cohomology -- a
contradiction.
\end{proof}
\begin{corollary}
\label{top2}
Let $(M,c,\nabla^W)$ be a K\"ahler-Weyl manifold which has complex
dimension $2$ or admits a pluricanonical metric. Let $(M',g')$
be a compact quotient of $H^m_{\mathbf K}$ and $u:M\to M'$ an
arbitrary map. Then:
\begin{itemize}
\item[(i)] if $\mathbf K=\mathbf R$, $u$ is trivial on homology in
  dimension greater than two.
\item[(iii)] if $\mathbf K=\mathbf H$, $u$ is trivial on homology in
  dimension greater than $2m$, one-half of the (real) dimension of $M'$.
\item[(iv)] if $\mathbf K=\mathbf O$, $u$ is trivial on homology in
  dimension greater than four. 
\end{itemize}
\end{corollary}
\begin{proof}
If a given map $u$ is homotopic to a constant map or a map onto a
closed geodesic, than the statements of the corollary are
trivial. Otherwise Theorem~\ref{ExTh} implies that $u$ is homotopic to
a Weyl-harmonic map and the statements follow from the rank estimates
in Corollary~\ref{rank2}.
\end{proof}

Note that as special case the corollary above contains the following
statement: {\it any map of a compact complex surface to a compact
  manifold of constant negative curvature induces the trivial map on
  homology in dimension greater than two. In particular, a compact
  $4$-dimensional manifold of constant negative curvature does not
  admit a complex structure.} The last part of the statement is also a
consequence of the results of Wall and Kotschick (corollary
of~\cite[Th.~10.5]{Wall}, corrected by~\cite[Prop.~2]{Kot}) and
Carlson and Toledo~\cite[Cor~1.3]{CT97}. 

In comparison with K\"ahler geometry, little is known about
topological obstructions to the existence of locally conformal
K\"ahler metrics (in complex dimension greater than
two).\footnote{Here we assume that any K\"ahler metric is, of course,
  locally conformal K\"ahler.} The following statement gives a
construction of manifolds which are not homotopy equivalent to
K\"ahler-Weyl manifolds with pluricanonical metrics.
\begin{corollary}
Let $M$ and $M'$ be manifolds of the same dimension and suppose that
$M'$ is a locally symmetric space of non-compact type.
\begin{itemize}
\item[(i)] Suppose that $M'$ is an orientable irreducible space whose
  universal cover is
$$
SO_0(p,q)/SO(p)\times SO(q),\quad Sp(p,q)/Sp(p)\times
Sp(q),\text{ or }F_{-4(20)}/\mathit{Spin}(9),
$$
then the connected sum $M\sharp M'$ is not homotopy equivalent to a
K\"ahler-Weyl manifold with a pluricanonical metric.
\item[(ii)] Suppose that $M'$ is locally Hermitian symmetric and
  $\pi_1(M)\ne 0$, then the connected sum $M\sharp M'$ is not homotopy
  equivalent to a K\"ahler-Weyl manifold with a pluricanonical metric.
\end{itemize}
\end{corollary}
\begin{proof}
Suppose the contrary and let $X$ be a K\"ahler-Weyl manifold with a
pluricanonical metric of the same homotopy type as the connected sum
$M\sharp M'$. Then there is a map of degree one from $X$ to $M'$,
obtained by applying the homotopy equivalence and then collapsing $M$
into a point. Now the statement of part~$(i)$ follows from
Cor.~\ref{top1} and~\ref{top2}. To prove part~$(ii)$, we use
Cor.~\ref{Euler} and Th.~\ref{Siu1} to conclude that the degree one
map is homotopic to a holomorphic map. Further, since it vanishes on
the subgroup $\pi_1(M)$, the contradiction follows from the following
observation of Jost and Yau~\cite[Lemm.~9]{JostYau}: a holomorphic map
of degree $\pm 1$ between compact complex manifolds is injective on
the fundamental group.
\end{proof}

\subsection{Co-compact lattices in $SO(1,n)$}
In this section we show that if a co-compact lattice in $SO(1,n)$ with
$n>2$ is the fundamental group of a compact K\"ahler-Weyl
manifold, then the latter can not be a complex surface and can not
admit a pluricanonical metric. This is an extension of the results by
Carlson and Toledo~\cite{CT89,CT97} which say that no such lattice can
be the fundamental group of a compact K\"ahler manifold or a compact
complex surface. The proof is based on the following factorisation
theorem, which is a sharpened version of the results due to Carlson
and Toledo~\cite{CT89} and Jost and Yau~\cite{JY2}.
\begin{FT}
Let $M$ be a complex manifold and $M'$ be a manifold of constant
negative curvature. Let $u:M\to M'$ be a pluriharmonic map such that
the rank of $du$ is at most two and equals two on an open and dense
subset of $M$. Then there exists a compact Riemannian surface $S$ and
a holomorphic map $h:M\to S$ and a harmonic map $\phi:S\to M'$ such
that $u=\phi\circ h$.
\end{FT}
\begin{proof}
{\it Step~1: Introducing a holomorphic foliation.}
Let $U$ be a coordinate ball in $M$ such that the rank of $du$ equals
two on $U$. It is straightforward to see that the rank of $du(x)$
equals two if and only if the image of $T^{1,0}_xM$ under $du(x)$ is
complex one-dimensional and contains no real vectors. Denote by
$z^\alpha$ the holomorphic coordinates on $U$ and by $\partial_\alpha$
the corresponding coordinate vector fields. Without loss of
generality, we can suppose that the image of the first coordinate
vector $d'u(\partial_1)$, denoted by $X$, does not vanish. Then the
images $d'u(\partial_\alpha)$ are spanned by $X$, 
$$
d'u(\partial_\alpha)=q_\alpha X,
$$
and the functions $q_\alpha$ are holomorphic. To demonstrate the
latter we use the hypothesis that $u$ is pluriharmonic:
\begin{multline*}
0=(\widetilde\nabla du)_{\alpha\bar\beta}=\nabla'_{\bar\beta}
(du(\partial_\alpha))=\nabla'_{\bar\beta}(q_\alpha X)
=(\partial_{\bar\beta}q_\alpha)X+q_\alpha\nabla'_{\bar\beta}
(du(\partial_1))=(\partial_{\bar\beta}q_\alpha)X.
\end{multline*}
Above $\nabla'$  and $\widetilde\nabla$ denote the natural connections
on the bundles $u^*T^cM'$ and bundle $\Hom(T^cM,u^*T^cM')$
respectively; the latter is determined by the torsion-free complex
connection on $M$ and the connection $\nabla'$ on $M'$. The
distribution given by the vectors $(\partial_\alpha-q_\alpha\partial_1)$, 
where $\alpha>1$, is a holomorphic kernel of $du$ and is closed under
the Lie bracket. By the complex Frobenius theorem, we obtain a
holomorphic foliation $\mathcal F$ on the set of $M\backslash\mathcal
X$ where $du$ has rank two, and the map $u$ is constant on its leaves. 

Note that for a Hermitian metric on $M$, the map $u$ solves the
corresponding Hermitian harmonic map equation. Choosing a
real-analytic metric, it follows that $u$ is real-analytic and the set
$\mathcal X$, where the rank of $du$ does not equal two, is a
real-analytic subvariety of $M$.

\medskip
\noindent
{\it Step~2: Extending the holomorphic foliation $\mathcal F$.}
We outline the arguments of Mok~\cite[Prop.~(2.2.1)]{Mok85} showing
that $\mathcal F$ extends to a holomorphic foliation on
$M\backslash\mathcal Z$, where $\mathcal Z$ is a complex analytic
subvariety of complex codimension at least two.

By the discussion in Sect.~\ref{PLURI}, since $u$ is pluriharmonic, we
have
$$
R'(X,Y)=0\quad\text{for all}\quad X,Y\in du(T_x^{1,0}M),
\quad x\in M.
$$
This in turn is equivalent to $(d^{\prime\prime}_{\nabla'})^2=0$,
where
$$
d^{\prime\prime}_{\nabla'}:A^{0,k}(M,u^*T^cM')\longrightarrow
A^{0,k+1}(M,u^*T^cM')
$$
is the (0,1)-part of the differential on the $(u^*T^cM')$-valued
forms, determined by the connections on $M$ and $M'$. The latter is
the integrability condition that allows to define the Koszul-Malgrange
complex structure on $u^*T^cM'$, see~\cite{KM}: a local section $s$ of
$u^*T^cM'$ is holomorphic if and only if the
$d^{\prime\prime}_{\nabla'}s=0$. This complex structure turns
$\Hom(T^{1,0}M,u^*T^cM')$ into a holomorphic vector bundle, and since
$u$ is pluriharmonic, $d'u$ is a holomorphic section of it.

Denote by ${\mathbf P}\Hom$ the projectivisation of $\Hom(T^{1,0}M,
u^*T^cM')$ and by $[d'u]$ the image of $d'u$ under the natural
projection
$$
\Hom(T^{1,0}M,u^*T^cM')\backslash\{\text{zero-section}\}
\longrightarrow {\mathbf P}\Hom.
$$
Clearly, $[d'u]$ is a holomorphic section of ${\mathbf P}\Hom$ over
$M\backslash\mathcal X$. Since $\mathcal X$ is real-analytic, 
by~\cite[Prop.~(2.2.2)]{Mok85} this section extends meromorphically to
$M$. Denote the extension by $\lambda$; it is holomorphic everywhere
except for the set of indeterminacies, which is a complex analytic
subvariety $\mathcal Z$ of complex codimension at least two. It is
then straightforward to show that the distribution
$$
D_x=\{v\in T^{1,0}_xM: \eta(v)=0\text{ for all }\eta\in
\Hom(T^{1,0}_xM,T^c_{u(x)}M')\text{ such that }[\eta]=\lambda\}
$$
is integrable and defines a holomorphic foliation on
$M\backslash\mathcal Z$, which is an extension of $\mathcal F$.

Note that $\mathcal Z$ is contained in the zero-set of $d'u$ and,
hence, in the zero-set of $du$. Thus, the map $u$ is constant on its
connected components.

\medskip
\noindent
{\it Step~3: Factorisation via a holomorphic equivalence relation.}
Consider the set
$$
V_0=\{(x,y)\in M\times M: u(x)=u(y)\}.
$$
It is a real-analytic subvariety of $M\times M$ and let $V_0=\cup V^k$
be its decomposition into irreducible components. Clearly, the
diagonal of $M\times M$ is contained in some branch $V^i$. Since
$\mathcal F$ is holomorphic, $V^i$ is complex analytic at any smooth
point and, by the theorem of Diederich-Forn\ae ss~\cite{DF}, is a
complex analytic subvariety. Thus, the set of the complex analytic
components $V^k$'s such that
$$
V^k\cap(M\backslash\mathcal X)\times (M\backslash\mathcal X)
\ne\varnothing
$$
is non-empty. Denote by $V$ their union and define
$$
\mathfrak R=V\cap(M\backslash\mathcal Z)\times (M\backslash\mathcal
Z).
$$
As in~\cite[Prop.~(2.2)]{Mok88} one shows that $\mathfrak R$ is the
graph of an open holomorphic equivalence relation on
$M\backslash\mathcal Z$. The quotient space $(M\backslash\mathcal
Z)/\mathfrak R$ is endowed with a natural structure sheaf by assigning
to every open subset $U$  the set of holomorphic functions on
$h^{-1}(U)$; here
$$
h:M\backslash\mathcal Z\longrightarrow (M\backslash\mathcal
Z)/\mathfrak R
$$
denotes the natural projection. By the result of Kaup~\cite{Kaup}, the
ringed space $(M\backslash\mathcal Z)/\mathfrak R$ is isomorphic to a
normal complex space, and the projection $h$ is holomorphic. Since the
generic fiber of $h$ has complex codimension one, we conclude that
$(M\backslash\mathcal Z)/\mathfrak R$ is one-dimensional. Since the
latter is also normal, it is necessarily a smooth Riemann surface. By
the Riemann extension theorem, the map $h$ extends holomorphically to
$M$, and we denote by $S$ its image -- a compact Riemann surface. By
the definition of $\mathfrak R$, the map $u$ factors through $h$ on
$M\backslash\mathcal Z$. Since the former is constant on the connected
components of $\mathcal Z$, this factorisation extends to $M$.

From the above we see that there exists a continuous map $\phi:S\to
M'$ such that $u=\phi\circ h$; it remains to show that $\phi$ is
harmonic. First, by the theorem of Eells and Sampson~\cite{EeSa}, the
map $\phi$ is homotopic to a harmonic map $\phi'$. Since $S$ is a
Riemannain surface, the map $\phi'$ is also Hermitian harmonic and,
since the map $h$ is holomorphic, the composition $\phi'\circ h$ is a
Hermitian harmonic map. Since the latter is homotopic to $u$ and the
rank of $du$ is greater than one generically, Theorem~\ref{UniTh}
implies that $u$ is a unique Hermitian harmonic map in its homotopy
class, and the maps $\phi$ and $\phi'$ coincide.
\end{proof}

Now we state the principal application.
\begin{theorem}
\label{lattices}
Let $\Gamma$ be a co-compact discrete subgroup of $SO(1,n)$ with
$n>2$. If $\Gamma$ is the fundamental group of a compact K\"ahler-Weyl
manifold, then the latter can not be a complex surface and can not
admit a pluricanonical metric. 
\end{theorem}
\begin{proof}
By passing to a subgroup of finite index, we may assume that $\Gamma$
is torsion-free. An Eilenberg-MacLane space $K(\Gamma,1)$ can be
constructed as the quotient $\Gamma\backslash D$, where
$D=SO(1,n)/SO(n)$ is the hyperbolic $n$-space. Let $M$ be a compact
K\"ahler-Weyl manifold whose fundamental group $\pi_1(M)$ is
isomorphic to $\Gamma$. The isomorphism is induced by a smooth map
$u:M\to\Gamma\backslash D$ which classifies the universal cover of
$M$. Since $\Gamma$ is co-compact and of cohomological dimension
greater than two, it can not be $\mathbf Z$ and the map $u$ is not
homotopic to a map onto a closed geodesic. Thus, by Theorem~\ref{ExTh}
we may assume that $u$ is a Weyl-harmonic map. Now suppose the
contrary to the statement of the theorem. Then Theorem~\ref{AP}
and Cor.~\ref{rank2} imply that $u$ is pluriharmonic and the rank
of $du$ is at most two. Since $u$ is not homotopic to a map onto a
closed geodesic, Proposition~\ref{rank<1} shows that the rank of $du$,
in fact, equals two on an open and dense subset of $M$. Now the
Factorisation theorem applies: the map $u$ factors as $\phi\circ h$,
where $h:M\to S$ and $S$ is a Riemannian surface. Since $u$ is
isomorphic on the fundamental groups, the homomorphism $h_*:\pi_1(M)
\to\pi_1(S)$ is injective and $\Gamma$ is identified with a subgroup
in $\pi_1(S)$. The latter acts freely on the universal cover of $S$,
which has to be contractible. Hence, the cohomological dimension of
$\Gamma$ is at most two. However, since $\Gamma\backslash D$ is a
$K(\Gamma,1)$-space, the cohomological dimension of $\Gamma$ is in
fact $n$. Since $n>2$, we are in the presence of a contradiction.
\end{proof}

\section{Proofs of Theorems~\ref{ExTh} and~\ref{UniTh}}
\label{ETproof}
The purpose of this section is to prove Theorem~\ref{ExTh}. Throughout
we suppose that $M'$ has non-positive sectional curvature. We start
with introducing some notation.

\subsection{Preliminaries}
Let $\tilde M$ and $\tilde M'$ be universal covers of manifolds $M$
and $M'$ respectively. Then the fundamental groups $\pi_1(M,\cdot)$
and $\pi_1(M',\cdot)$ act by isometries on them such that 
$$
M=\tilde M/\pi_1(M,\cdot)\quad\text{and}\quad M'=\tilde
M'/\pi_1(M',\cdot).
$$
The distance function $\tilde r:\tilde M'\times\tilde M'\to \mathbf R$
is well-defined and smooth outside the diagonal; by $r$ we denote the
induced function on the quotient $(\tilde M'\times\tilde
M')/\pi_1(M',\cdot)$.

Now let $u$ and $v$ be homotopic maps from $M$ to $M'$ and $H$ be a
homotopy between them. Its lifting to the universal covers defines the
liftings $\tilde u$ and $\tilde v$ of these maps. Further, the map
$\tilde w:\tilde M\to\tilde M'\times\tilde M'$, given by $x\mapsto
(\tilde u(x),\tilde v(x))$, is equivariant with respect to the action
of the fundamental groups and, hence, descends to a map
$$
w:M\longrightarrow (\tilde M'\times\tilde M')/\pi_1(M',\cdot).
$$
Finally, we define the function $\rho_H(u,v)$ as the composition
$r\circ w$. The function $\rho^2_H(u,v)$ is smooth on $M$, and we give
an inequality for its Weyl laplacian below (Lemma~\ref{pseudoLaIn}). 
First, we introduce some more notation.

For a point $(u,v)\in\tilde M'\times\tilde M'$ choose an orthonormal
basis $e_1,\ldots,e_{n'}$ for $T_u\tilde M'$. By parallel transport
along the shortest geodesics from $u$ to $v$ we also obtain a basis
$\bar e_1,\ldots,\bar e_{n'}$ for $T_v\tilde M'$ and consider the
frame $e_1,\ldots,e_{n'},\bar e_1,\ldots,\bar e_{n'}$ as a basis for
$T_u\tilde M'\times T_u\tilde M'$. Let $\omega^1,\ldots,\omega^n$ be
an orthonormal basis for $T_x^*M$. Then for maps $u$ and $v$ their
differentials $du(x)$ and $dv(x)$ at a point $x\in M$ can be
decomposed as $u^i_\alpha e_i\otimes\omega^\alpha$ and
$v^{\bar\imath}_\alpha e_{\bar\imath}\otimes\omega^\alpha$
respectively. Analogously, the second fundamental forms have the
coefficients $u^i_{\alpha\beta}$ and
$v^{\bar\imath}_{\alpha\beta}$. In particular, the tensions fields
satisfy the relations
$$
\tau^i(u)=\sum_\alpha u^i_{\alpha\alpha}\quad\text{and}\quad
\tau^{\bar\imath}(v)=\sum_\alpha v^{\bar\imath}_{\alpha\alpha}.
$$
The following statement is a pseudo-version of the calculation due to
Schoen and Yau~\cite{ShYau}. The symbol $\Delta^W$ denotes the
Weyl laplacian, given by $\mathit{trace}_g\nabla^W d$.
\begin{lemma}
\label{pseudoLaIn}
Suppose $M'$ has non-positive sectional curvature and $M$ is endowed
with a Weyl connection $\nabla^W$. Then for any homotopic maps
$u,v:M\to M'$ and a given homotopy $H$ between them the Weyl laplacian
of the function $\rho^2=\rho^2_H(u,v)$ satisfies the following
inequality
\begin{equation}
\label{SYineq}
\Delta^W\rho^2\geqslant 2\sum_{i,\alpha}(u^i_\alpha-
v^{\bar\imath}_\alpha)^2 -2\rho(\abs{\tau^W(u)}+
\abs{\tau^W(v)}).
\end{equation}
\end{lemma}
\begin{proof}
First, we obviously have
$$
(\rho^2)_\alpha=2\rho\rho_\alpha=2\rho(r_i u^i_\alpha+r_{\bar\imath}
v^{\bar\imath}_\alpha).
$$
It is a straightforward calculation, cf.~\cite[p.368]{ShYau}, to show the
following relations
\begin{multline}
\label{PseudoDist}
\Delta^W\rho^2=\sum_\alpha(\rho^2)_{\alpha\alpha}-((n-2)/2)(\rho^2)_\alpha
\Theta^\alpha=\sum_\alpha (r^2)_{X_\alpha X_\alpha}+2\rho
\sum_\alpha(r_i u^i_{\alpha\alpha}+r_{\bar\imath} 
v^{\bar\imath}_{\alpha\alpha})\\-(n-2)\rho(r_i u^i_{\alpha}+r_{\bar\imath} 
v^{\bar\imath}_{\alpha})\Theta^\alpha=\sum_\alpha (r^2)_{X_\alpha
  X_\alpha}+2\rho(r_i\tau^W(u)^i+r_{\bar\imath}
\tau^W(v)^{\bar\imath}),
\end{multline}
where the vectors $X_\alpha=u^i_\alpha e_i+v^{\bar\imath}_\alpha
e_{\bar\imath}$ and the summation convention for the repeated indices
is used troughout. As is shown in~\cite[p.365]{ShYau},
$$
\sum_{\alpha}(r^2)_{X_\alpha X_\alpha}\geqslant 2\sum_{i,\alpha}
(u^i_\alpha-v^{\bar\imath}_\alpha)^2,
$$
if $M'$ has non-positive sectional curvature. Finally, choose the
orthonormal frame $\{e_i\}$ such that $e_1$ is tangent to the shortest
geodesic joining $u$ and $v$. Then the first variation formula implies
that $\nabla r=\bar e_1-e_1$. This yields the following relation
$$
2\rho(r_i\tau^W(u)^i+r_{\bar\imath}\tau^W(u)^{\bar\imath})
\geqslant-2\rho(\abs{\tau^W(u)}+\abs{\tau^W(u)}).
$$
Combining the last two inequalities with the formula for
$\Delta^W\rho^2$ we demonstrate the lemma.
\end{proof}

The next assertion follows directly from the inequality above by
integration by parts in the left-hand side.
\begin{corollary}
\label{energy&dist}
Under the conditions of Lemma~\ref{pseudoLaIn}, we have the following
inequality
$$
E(u)\leqslant 2E(v)+C\int_M\rho^2d\mathit{Vol}_g+\int_M\rho
(\abs{\tau^W(u)}+\abs{\tau^W(v)})d\mathit{Vol}_g,
$$
where the positive constant $C$ depends only on the metric $g$ and the
Higgs field $\Theta^\sharp$ on $M$ and their derivatives.
\end{corollary}

\subsection{Key estimates}
We prove the existence of Weyl harmonic maps via considering the
corresponding heat flow
\begin{equation}
\label{PseudoHF}
\frac{\partial}{\partial t}u(t,x)=\tau^W(u)(t,x),
\qquad u(0,x)=u^0(x),\qquad x\in M, ~t\in[0,+\infty);
\end{equation}
where $u^0$ is a continuous map in a given homotopy class. We show
that a solution of this parabolic equation exists for all
$t\in[0,+\infty)$ -- we assume throughout that the sectional curvature
  of $M'$ is non-positive. Moreover, if $u^0$ belongs to a homotopy
  class which satisfies suppositions of Theorem~\ref{ExTh}, then there
  exists a sequence $t_n\to +\infty$ such that $u(t_n,\cdot)$
  converges smoothly to a Weyl harmonic map.

Linearising this parabolic equation and using results on linear
parabolic systems and the implicit function theorem, it follows in a
standard manner that equation~\eqref{PseudoHF} has a solution for
small $t$ and, by the semi-group property, the interval of existence
in $[0,+\infty)$ is open. To show that it is closed, and hence the
  solution exists for all $t$, we first state the following claims.
\begin{claim}
\label{c1}
Let $u(t,x)$ be a solution of equation~\eqref{PseudoHF} on the
interval $[0,T)$. Then the quantity $\sup_{x\in M}\norm{(\partial
    u/\partial t)}^2$ is non-increasing in $t\in [0,T)$.
\end{claim}
\begin{proof}
Since the sectional curvature of $M'$ is non-positive, the second
Bochner identity in Lemma~\ref{Bpar} implies that
$$
\frac{1}{2}\left(\Delta^W-\frac{\partial}{\partial t}\right)
\norm{(\partial u/\partial t)}^2\geqslant 0.
$$
Now the claim follows from the parabolic maximum principle.
\end{proof}
\begin{claim}
\label{c2}
Let $u(t,x)$ be a solution of equation~\eqref{PseudoHF} on the
interval $[0,T)$. Then it satisfies the following point-wise bound
$$
\norm{du(t,x)}^2\leqslant C_1\sup_{\tau\leqslant t}
\int_M\norm{du(\tau,x)}^2d\mathit{Vol}_g(x),
$$
for $\delta\leqslant t<T$, where $\delta>0$ is sufficiently
small; the positive constant $C_1$ depends on the geometry of $M$, the
Higgs form $\Theta$, and $\delta$.
\end{claim}
\begin{proof}
Since the sectional curvature of $M'$ is non-positive, the first
Bochner identity in Lemma~\ref{Bpar} implies that
$$
\frac{1}{2}\left(\Delta^W-\frac{\partial}{\partial t}\right)
\norm{du(t,x)}^2\geqslant -C_*\norm{du(t,x)}^2,
$$
where $du(t,x)$ denotes the differential with respect to $x\in M$ and
the positive constant $C_*$ depends on the bound for the Ricci
curvature of $M$ and the $C^1$-norm of the Higgs form $\Theta$. Now a
standard estimate for parabolic inequalities implies the claim. More
precisely, for any sufficiently small $R>0$ and $\delta>0$ we have the
following estimate
$$
\sup_{\stackrel{[t-\delta/2,t]}{B(x,R/2)}}\norm{du(t,x)}^2\leqslant
C R^{-(n+2)/2}\left(1+\delta^{-1/2}R\right)
\int_{t-\delta}^t\int_M\norm{du(\tau,x)}^2\mathit{dVol}\,
\mathit{d\tau},
$$
where $t\geqslant\delta$; for the details we refer
to~\cite[Ch.~III]{LSU}.
\end{proof}
\begin{lemma}
\label{existenceHF}
Let $(M,g)$ and $(M',g')$ be closed Riemannian manifolds. Suppose that
$M'$ has non-positive sectional curvature and $M$ is endowed with a
Weyl connection $\nabla^W$. Then for a continuous map $u^0:M\to M'$
there exists a solution of equation~\eqref{PseudoHF} for any positive
time $t$.
\end{lemma}
\begin{proof}
As was mentioned above the set formed by the $t$'s for which a
solutions exists is open and non-empty. By $\rho^2(t,x)=\rho^2(u(t,x),
u^0(x))$ we denote the squared distance along a geodesic homotopy
between $u(t,x)$ and $u^0(x)$ such that the corresponding geodesics
are homotopic to the paths given by the heat flow. Now the combination
of Claim~\ref{c1} and Corollary~\ref{energy&dist} implies the estimate
$$
\int_M\norm{du(t,x)}^2d\mathit{Vol}_g(x)\leqslant \bar C_1\int_M
\rho^2(t,x)d\mathit{Vol}_g(x)+\bar C_2.
$$
By Claim~\ref{c2}, we further obtain
\begin{equation}
\label{C^1dist}
\norm{du(t,x)}^2\leqslant\bar C_3\sup_{\tau\leqslant t}\sup_{x\in
  M}\rho^2(\tau,x)+\bar C_4,
\end{equation}
where $t$ is greater than some $\delta>0$. Note also that
$$
\rho^2(\tau,x)\leqslant\tau^2\sup_{s\in[0,\tau]}\norm{(\partial
  u/\partial t)(s,x)}\leqslant \bar C_5\tau^2;
$$
in the last inequality we used Claim~\ref{c1}. Combining this with the
previous relation, we arrive at the upper bound
$$
\norm{du(t,x)}^2\leqslant C(1+t^2)
$$
with some positive constant $C$ which is independent of $t$ and
$x$. Since by Claim~\ref{c1} we also have an upper bound on
$\norm{(\partial u/\partial t)(t,x)}$, regularity theory for linear
(elliptic and parabolic) equations yields $C^{2,\alpha}$-estimates for
a solution of equation~\eqref{PseudoHF}. This implies that the set of
the $t$'s where a solution exists is closed and, hence, a solution
exists globally.
\end{proof}

The next two estimates are concerned with the values of $t$ greater
than some $\delta>0$.
\begin{claim}
\label{c3}
A solution $u(t,x)$ of equation~\eqref{PseudoHF} satisfies the
following estimate
$$
\sup_{x\in M}\rho^2(t,x)\leqslant C_2\sup_{\tau\leqslant t}
\left(\inf_{x\in M}\rho^2(\tau,x)+\sup_{x\in M}\rho(\tau,x)\right).
$$
Here $\rho(t,x)=\rho(u(t,x),u^0(x))$ denotes the distance along
geodesics which are homotopic to the paths given by the heat flow. The
positive constant $C_2$ depends on the $C^2$-norm of $u^0$, a bound
for $\norm{(\partial u/\partial t)}$, and the same quantities as the
constant $C_1$ in Claim~\ref{c2}.
\end{claim}
\begin{proof}
First, the combination of Lemma~\ref{pseudoLaIn} and Claim~\ref{c1}
yields the inequality
$$
\Delta^W \rho^2(t,x)\geqslant -C\rho(t,x).
$$
Now let $x_0$ be a point from $M$ where the function $\rho(t,\cdot)$
achieves its minimum. Denote by $R$ a positive real number which is
less than the injectivity radius of $M$. Due to the maximum principle
for elliptic inequalities on the domains $B=B(x_0,R)$ and $M\backslash
B$, see~\cite[Sect.~3.3]{GT}, we obtain
$$
\sup_M\rho^2(t,\cdot)\leqslant\sup_{\partial B}\rho^2(t,\cdot)+\bar
C_6\sup_M\rho(t,\cdot).
$$
It is a straightforward calculation to show that
$$
\sup_{\partial B}\rho^2(t,\cdot)\leqslant\rho^2(t,x_0)+2R\sup_{B}
\left[\rho(t,\cdot)(\norm{du(t,\cdot)}+\norm{du^0})\right].
$$
Combining these inequalities with relation~\eqref{C^1dist} in the
proof of Lemma~\ref{existenceHF}, we get
$$
\sup_M\rho^2(t,\cdot)\leqslant\rho^2(t,x_0)+\bar
C_7\sup_M\rho(t,\cdot)+2R\bar C_3\sup_{\tau\leqslant t}\sup_M
\rho^2(\tau,\cdot).
$$
Now choosing $R$ sufficiently small we demonstrate the claim.
\end{proof}
\begin{claim}
\label{c4}
A solution $u(t,x)$ of equation~\eqref{PseudoHF} satisfies the
following point-wise estimate
$$
\norm{du(t,x)}\leqslant C_3\sup_{\tau\leqslant t}\sup_{x\in M}
\rho(\tau,x)^{1/2}+C_4,
$$
where $\rho(t,x)=\rho(u(t,x),u^0(x))$ is the geodesic distance as
above and positive constants $C_3$ and $C_4$ depend on the same
quantities as $C_2$ in Claim~\ref{c3}.
\end{claim}
\begin{proof}
Let $t_*\in[\delta,t]$ be a point where $\sup_M\rho(\tau,\cdot)$
achieves its maximum. Replacing the quantity $\rho^2$ in the left-hand
side of inequality~\eqref{SYineq} by
$$
\bar\rho^2(\cdot)=\rho^2(t_*,\cdot)-C_2\sup_{\tau\leqslant t}
\inf_{x\in M}\rho^2(\tau,x),
$$
we, as in Corollary~\ref{energy&dist}, arrive at the following:
$$
\int_M\norm{du(t_*,\cdot)}^2d\mathit{Vol}_g\leqslant\bar C_8
\int_M\bar\rho^2d\mathit{Vol}_g+\bar C_9\int_M\rho d\mathit{Vol}_g
+\bar C_{10}.
$$
By Claim~\ref{c3} the first integral on the left-hand side is
estimated by
$$
C_2\mathit{Vol}(M)\sup_{\tau\leqslant t}\sup_{x\in M}\rho(\tau,x).
$$
Now the use of Claim~\ref{c2}, combined with elementary
transformations, yields the claimed inequality.
\end{proof}

\subsection{The proof of Theorem~\ref{ExTh}}
The following auxiliary assertion we state as a lemma.
\begin{lemma}
\label{std}
Suppose that the conditions of Theorem~\ref{ExTh} hold. Let $u(t,x)$,
where $t\geqslant 0$ and $x\in M$, be a solution of
equation~\eqref{PseudoHF} such that the norm $\norm{du(t,x)}$ is
bounded uniformly in $t\geqslant 0$ and $x\in M$. Then there exists a
sequence of times $t_n\to +\infty$ such that $u(t_n,\cdot)$ converges
in $C^2$-topology to a Weyl harmonic map homotopic to $u^0$.
\end{lemma}
\begin{proof}
First, recall the Bochner identity, Lemma~\ref{Bpar}:
\begin{equation}
\label{B}
\frac{1}{2}\left(\Delta^W-\frac{\partial}{\partial t}\right)
\kappa=\norm{\nabla (\partial u/\partial t)}^2-\sum_{\alpha}
\langle R'(du\cdot e_\alpha,(\partial u/\partial t))du\cdot 
e_\alpha,(\partial u/\partial t)\rangle_{g'},
\end{equation}
where $\kappa(t,x)\geqslant 0$ denotes the squared norm
$\norm{(\partial u/\partial t)}^2$. Without loss of generality we may
assume that the metric $g$ on $M$ is Gauduchon, i.~e. the
corresponding Higgs field $\Theta$ is co-closed. This implies that the
integral of $\Delta^W\!\!\kappa$ over $M$ vanishes. Thus integrating
identity~\eqref{B} and using the curvature hypothesis we conclude,
that there exists a sequence of times $t_n\to +\infty$ such that each
term in the right-hand side in~\eqref{B} converges to zero point-wise
along $t_n$.

Now we return to the solution $u(t,x)$ of the heat flow equation.
By regularity theory for linear (elliptic and parabolic) equations,
the bound on $\norm{du(t,x)}$ together with Claim~\ref{c1} imply, in a
standard manner, $C^{2,\alpha}$-bounds on $u(t,\cdot)$ and $(\partial
u/\partial t)(t,\cdot)$ independent of $t\geqslant 0$. Since the
embeddings $C^{2,\alpha}\subset C^2$ are compact, we obtain a
subsequence of the $t_n$'s (denoted by the same symbol) such that
$u(t_n,\cdot)$ and $(\partial u/\partial t)(t_n,\cdot)$ converge in
$C^2$-topology to a map $u_\infty(\cdot)$ and a vector field $\matheur
v(\cdot)$ respectively. Moreover, since any term in the right-hand
side in~\eqref{B} converges to zero along $t_n$, we conclude that
$\matheur v$ is a parallel vector field along $u_\infty$. By the
suppositions of Theorem~\ref{ExTh}, part~$(i)$, it vanishes and
passing to the limit in equation~\eqref{PseudoHF}, we obtain that
$u_\infty$  is a Weyl harmonic map. The vanishing of the curvature
term in~\eqref{B} implies the same conclusion under the hypotheses in
the part~$(ii)$.
\end{proof}

For a proof of the theorem it is sufficient to show that the quantity
$\rho(u(t,x),u^0(x))$ is bounded uniformly in $t\geqslant 0$ and $x\in
M$. Indeed, in this case Claim~\ref{c4} implies the bound on
$\norm{du(t,x)}$ and an application of Lemma~\ref{std} yields a
Weyl harmonic map. Suppose the contrary. Then there exists a
sequence $t_n\to +\infty$ such that
$$
\rho(t_n,x)=\rho(u(t_n,x),u^0(x))\to +\infty,\qquad\text{as}\quad n\to
+\infty,
$$
for some $x\in M$. Moreover, Claim~\ref{c3} implies that $\rho(t_n,x)\to
+\infty$ for any $x\in M$. Without loss of generality, we can suppose
that the sequence $\sup_{M}\rho(t_n,\cdot)$ is non-decreasing.

Denote by $\gamma^n_x$ a geodesic joining points $u^0(x)$ and
$u(t_n,x)$ as always in the homotopy class determined by the heat flow
homotopy. Let $T_n$ be the maximum of their lengths as $x$ ranges over
$M$ and suppose that all paths $\gamma^n_x$ are parameterised by
$\tau\in[0,T_n]$; in other words, $\gamma^n_x(0)=u^0(x)$ and
$\gamma^n_x(T_n)=u(t_n,x)$. The sequence of these $T_n$'s converges to
infinity and, after a selection of subsequence, $\gamma^n_x$ converges
to a geodesic ray $\gamma_x$. In sequel, we also suppose that the
sequence $T_n$ is not decreasing.

Consider the function $d(\gamma^n_{x_1}(\tau),\gamma^n_{x_2}(\tau))$
for arbitrary $x_1$ and $x_2\in M$, where the distance is measured in
some fixed homotopy class of arcs connecting $\gamma^n_{x_1}(\tau)$
and $\gamma^n_{x_2}(\tau)$. Claim~\ref{c4} implies the estimate
$$
d(\gamma^n_{x_1}(T_n),\gamma^n_{x_2}(T_n))\leqslant C\sup_M\rho(t_n,
\cdot)^{1/2}+C'.
$$
The right-hand side can be estimated by $\bar CT_n^{1/2}$ for a
sufficiently large $n$, and we arrive at the following
inequality
\begin{equation}
\label{aux1}
d(\gamma^n_{x_1}(T_n),\gamma^n_{x_2}(T_n))\leqslant \bar CT_n^{1/2}.
\end{equation}
Further, since $M'$ has non-positive sectional curvature, the function
$$
\tau\longmapsto d(\gamma^n_{x_1}(\tau),\gamma^n_{x_2}(\tau))
$$
is convex; see (the proof of) Lemma~\ref{pseudoLaIn}. This implies
that for a fixed $\tau\geqslant 0$
\begin{multline*}
\lim_{n\to\infty}d(\gamma^n_{x_1}(\tau),\gamma^n_{x_2}(\tau))\leqslant
\lim_{n\to\infty}\left[\left(1-\frac{\tau}{T_n}\right)d(\gamma^n_{x_1}(0),
\gamma^n_{x_2}(0))+\frac{\tau}{T_n}d(\gamma^n_{x_1}(T_n),
\gamma^n_{x_2}(T_n))\right]\\
\leqslant\lim_{n\to\infty}\left[\left(1-\frac{\tau}{T_n}\right)
d(\gamma^n_{x_1}(0),\gamma^n_{x_2}(0))+\frac{\tau}{T_n}\bar C T_n^{1/2}
\right]=d(\gamma_{x_1}(0),\gamma_{x_2}(0)).
\end{multline*}
The second inequality above follows from relation~\eqref{aux1}. Thus,
the limiting rays $\gamma_{x_1}$ and $\gamma_{x_2}$ have to satisfy
the relation
\begin{equation}
\label{aux2}
d(\gamma_{x_1}(\tau),\gamma_{x_2}(\tau))\leqslant 
d(\gamma_{x_1}(0),\gamma_{x_2}(0))\qquad\text{for
  every}\quad\tau\geqslant 0.
\end{equation}
For any positive $\tau$ define a map $u^\tau:M\to M'$ by the rule
$u^\tau(x)=\gamma_x(\tau)$. Then relation~\eqref{aux2} implies that
$$
\norm{du^\tau(x)}\leqslant\norm{du^0(x)},\qquad x\in M.
$$
Now suppose that $u^0$ is a harmonic map; it exists in any homotopy
class due to Eells and Sampson~\cite{EeSa}. Since it is energy
minimising the inequality above shows that every $u^\tau$ has to be
also harmonic. Now we are in the presence of a contradiction: under
the suppositions of the theorem, in both cases~$(i)$ and~$(ii)$, the
harmonic map is unique in the homotopy class under consideration;
see~\cite[Th.~1-2]{ShYau}.
\qed

\subsection{The proof of Theorem~\ref{UniTh}}
We follow the method of Schoen and Yau~\cite{ShYau}. Let $H$ be a
homotopy between pseudo-harmonic maps $u$ and $v$ and
$\rho^2=\rho^2_H(u,v)$ be the corresponding squared distance
function. By Lemma~\ref{pseudoLaIn}, the latter satisfies the
inequality $\Delta^W\rho^2\geqslant 0$ and, hence by the maximum
principle, is constant. From relation~\eqref{PseudoDist} we conclude
that $(r^2)_{X_\alpha X_\alpha}$ vanishes for any
$1\leqslant\alpha\leqslant n$. Lifting to the universal covers we
obtain that $({\tilde r}^2)_{\tilde X_\alpha\tilde X_\alpha}=0$, where
$\tilde X_\alpha$ denote the vectors $\tilde u^i_\alpha e_i+\tilde
v^i_\alpha e_{\bar\imath}$. Now~\cite[Prop.~1]{ShYau} applies and we
find vector fields $W_\alpha$ which are parallel along the unique
geodesic $\tilde\gamma_x$ joining the points $\tilde u(x)$ and $\tilde
v(x)$ and are such that
$$
W_\alpha(\tilde u(x))=\tilde u^i_\alpha e_i(\tilde
u(x))\quad\text{and}\quad
W_\alpha(\tilde v(x))=\tilde v^i_\alpha e_i(\tilde v(x)).
$$
Since $\tilde r$ is constant, we can parameterise each
$\tilde\gamma_x$ on $[0,1]$ proportional (independent on $x\in\tilde
M$) to arclength. We define a one-parameter family of maps $u_s:\tilde
M\to\tilde M'$ by setting $\tilde u_s(x)=\tilde\gamma_x(s)$. Then, the
parallelism of the $W_\alpha$'s implies that
$$
W_\alpha(\gamma_x(s))=(\tilde u_s)^i_\alpha e_i(\tilde u_s(x)),
\qquad s\in [0,1],\quad x\in M.
$$
Besides, the maps $\tilde u_s$ descend to the maps $u_s:M\to M'$ such
that $u_0=u$ and $u_1=v$.

Now we explain why each $u_s$ is a Weyl harmonic map. If it is not
than by Lemma~\ref{existenceHF} we can find a deformation $u_{s,t}$,
where $u_{s,t}$ is a solution of the heat flow with initial map
$u_s$. The same argument as in the proof of Lemma~\ref{Bpar} yields
the formula
\begin{multline*}
\frac{1}{2}\left(\Delta^W-\frac{\partial}{\partial t}\right)
\norm{(\partial u_s/\partial s)}^2=\norm{\nabla (\partial u_s/\partial
  s)}^2\\-\sum_{\alpha}\langle R'(du_s\cdot e_\alpha,(\partial u_s/\partial
  s))du_s\cdot e_\alpha,(\partial u_s/\partial s)\rangle,
\end{multline*}
where $e_\alpha$ is an orthonormal basis for $T_xM$. In particular,
$$
\frac{1}{2}\left(\Delta^W-\frac{\partial}{\partial t}\right)
\norm{(\partial u_s/\partial s)}^2\geqslant 0,\qquad t\in [0,+\infty),
\quad s\in [0,1],
$$
and the maximum principle shows that the length of the curve $s\mapsto
u_{s,t}(x)$ for any $t>0$ and $x\in M$ is not greater than the length
of the geodesic $s\mapsto u_s(x)$. Since the curves are homotopic with
fixed end-points, the former has to be also a geodesic and, moreover,
does not depend on $t$ and coincides with the curve $u_s(x)$. Now from
the heat flow equation, we conclude that each $u_s$ is a Weyl harmonic
map.

The statement on the parallelism of $(\partial u_s/\partial s)$
follows from the observation that the operator field $du_s(x)$ is
parallel along each geodesic $\gamma_x$. Indeed, for an orthonormal
basis $\{ e_\alpha\}$ for $T_xM$, each vector field $du(e_\alpha)$
coincides with $W_\alpha$ and is parallel along $\gamma_x$. In
particular, we see that the energy density
$$
e(du_s)(x)=\sum_\alpha\abs{du_s(e_\alpha)}^2(x)
$$
is constant along $\gamma_x$.

Finally, note that the vectors $W_\alpha$ span the image of
$du_s$. When $M'$ has negative sectional
curvature~\cite[Prop.~1]{ShYau} implies that the $W_\alpha$'s are
proportional to the tangent vector $\dot\gamma_x$ and, hence, the rank
of $du_s(x)$ is not greater than one for each $x\in M$ and $s\in
[0,1]$. This yields the last statement of the theorem.
\qed

\appendix
\section{Appendix: Bochner-type identities for the pseudo-harmonic maps
  and solutions of the corresponding heat flow.}
In this appendix we collected the Bochner-type identities for the
pseudo-harmonic setting, which are important for the analysis
above. First, we consider the solutions of the elliptic equation
\begin{equation}
\label{pseudoHME}
\tau^W(u)(x)=0,\qquad x\in M.
\end{equation}
As above the symbol $\Delta^W$ stands for  the Weyl laplacian,
$\mathit{trace}_g\nabla^W d$.
\begin{lemma}
\label{Bell}
Let $(M,g)$ and $(M',g')$ be arbitrary Riemannian manifolds and
suppose that $M$ is endowed with a Weyl connection $\nabla^W$
preserving $g$. Then for any smooth solution of
equation~\eqref{pseudoHME} the following relation holds
\begin{multline*}
\frac{1}{2}\Delta^W\norm{du}^2=\norm{\nabla du}^2-\sum_{\alpha,\beta}
\langle R'(du\cdot e_\alpha,du\cdot e_\beta)du\cdot e_\alpha,du\cdot
e_\beta\rangle_{g'}\\
+\sum_{\alpha}\left[\mathit{Ricci}(X_\alpha,X_\alpha)+\frac{n-2}{2}\left(
\nabla_{X_\alpha}\Theta\right)(X_\alpha)\right],
\end{multline*}
where $X_\alpha=(u^*\phi^\alpha)^\sharp$ and the systems
$\{e_\alpha\}$ and $\{\phi^\alpha\}$ are orthonormal bases in $T_\cdot
M$ and $T_{u{\scriptscriptstyle (\cdot)}}^*M'$ respectively at the point under
consideration; the symbol $\mathit{Ricci}$ denotes the Ricci
curvature of $M$ and $R'$ stands for the curvature tensor of $M'$.
\end{lemma}
\begin{proof}
First, recall the general Bochner formula for any smooth map $u:M\to
M'$, see~\cite{Jo},
\begin{multline}
\label{Bochner}
\frac{1}{2}\Delta\norm{du}^2=\norm{\nabla du}^2-\sum_{\alpha,\beta}
\langle R'(du\cdot e_\alpha,du\cdot e_\beta)du\cdot e_\alpha,du\cdot
e_\beta\rangle_{g'}\\
+\sum_{\alpha}\mathit{Ricci}(X_\alpha,X_\alpha)+\sum_\alpha\langle
\nabla'_{e_\alpha}\tau(u),du\cdot e_\alpha\rangle_{g'},
\end{multline}
where $\nabla$ and $\mathit{Ricci}$ stand for the Levi-Civita
connection of $g$ and its Ricci curvature respectively. Since the map
$u$ is pseudo-harmonic, relation~\eqref{Weyl&HME} implies that the
tension field $\tau(u)$ is equal to $c_ndu(\Theta^\sharp)$, where the
constant $c_n$ equals $((n-2)/2)$. Hence, for the last term in
identity~\eqref{Bochner} we obtain the following:
$$
\sum_\alpha\langle\nabla'_{e_\alpha}\tau(u),du\cdot e_\alpha
\rangle_{g'}=c_n\sum_\alpha\langle(\widetilde\nabla_{e_\alpha}du)\cdot
\Xi^\sharp,du\cdot e_\alpha\rangle_{g'}+c_n\sum_\alpha\langle du\cdot 
(\nabla_{e_\alpha}\Xi^\sharp),du\cdot e_\alpha\rangle_{g'}.
$$
It is a straightforward calculation to show that the first term in
this sum together with the quantity $(1/2)\Delta\norm{du}^2$ in the
left-hand side in~\eqref{Bochner} gives exactly the Weyl laplacian
$(1/2)\Delta^W\norm{du}^2$. The second term can be further
transformed as
$$
c_n\sum_{\alpha,\beta}\left(\nabla_{e_\alpha}\Xi\right)(e_\beta)\langle 
du\cdot e_\alpha,du\cdot e_\beta\rangle_{g'}=c_n\sum_{\alpha}\left(
\nabla_{X_\alpha}\Xi\right)(X_\alpha).
$$
Now the lemma follows by the combination of this with the Bochner
formula~\eqref{Bochner}.
\end{proof}
The next lemma is concerned with solutions of the pseudo-harmonic map
heat flow, the corresponding parabolic equation;
\begin{equation}
\label{pseudoHF}
\frac{\partial}{\partial t}u(t,x)=\tau^W(u)(t,x),\quad
u(0,x)=u^0(x),\qquad x\in M,~ t\in [0,+\infty).
\end{equation}
\begin{lemma}
\label{Bpar}
Let $(M,g)$ and $(M',g')$ be arbitrary Riemannian manifolds and
suppose that $M$ is endowed with a Weyl connection $\nabla^W$
preserving the conformal class of $g$. Then for any smooth solution of
equation~\eqref{pseudoHF} the following relations hold
\begin{multline*}
\frac{1}{2}\left(\Delta^W-\frac{\partial}{\partial t}\right)
\norm{du}^2=\norm{\nabla du}^2-\sum_{\alpha,\beta}
\langle R'(du\cdot e_\alpha,du\cdot e_\beta)du\cdot e_\alpha,du\cdot
e_\beta\rangle_{g'}\\
+\sum_{\alpha}\left[\mathit{Ricci}(X_\alpha,X_\alpha)+\frac{n-2}{2}\left(
\nabla_{X_\alpha}\Theta\right)(X_\alpha)\right],
\end{multline*}
$$
\frac{1}{2}\left(\Delta^W-\frac{\partial}{\partial t}\right)
\norm{(\partial u/\partial t)}^2=\norm{\nabla (\partial u/\partial
  t)}^2-\sum_{\alpha}\langle R'(du\cdot e_\alpha,(\partial u/\partial
  t))du\cdot e_\alpha,(\partial u/\partial t)\rangle_{g'},
$$
where $X_\alpha=(u^*\phi^\alpha)^\sharp$ and the systems
$\{e_\alpha\}$ and $\{\phi^\alpha\}$ are orthonormal bases in $T_\cdot
M$ and $T_{u{\scriptscriptstyle (\cdot)}}^*M'$ respectively at the point under
consideration; the symbol $\mathit{Ricci}$ denotes the Ricci
curvature of $M$ and $R'$ stands for the curvature tensor of $M'$.
\end{lemma}
\begin{proof}
The first relation follows in a similar fashion as in the proof of
Lemma~\ref{Bell}. Now we demonstrate the second relation. Standard
calculations imply the following Bochner-type formula for the family
of mappings $u(t,\cdot)$
\begin{multline*}
\frac{1}{2}\Delta\norm{(\partial u/\partial t)}^2=\norm{\nabla
(\partial u/\partial t)}^2-\sum_{\alpha}\langle R'(du\cdot e_\alpha,
(\partial u/\partial t))du\cdot e_\alpha,(\partial u/\partial t)
\rangle_{g'}\\
+\langle(\partial u/\partial t),\nabla'_{\partial/\partial t}\tau(u)
\rangle_{g'},
\end{multline*}
cf.~\cite{Sh}. Now the claim follows by substituting in this formula
$(\partial u /\partial t)+c_ndu(\Theta^\sharp)$ for the tension field
$\tau(u)$ and making elementary transformations.
\end{proof}

\section{Appendix: Unique continuation of Hermitian harmonic maps.}
The purpose of this appendix is to explain the following statement:
\begin{lemma}
Let $u:M\to M'$ be a Hermitian harmonic map between arbitrary
(connected) complex manifolds. Suppose that $u$ is holomorhic on a
non-empty open subset in $M$. Then $u$ is holomorphic on $M$.
\end{lemma}
\begin{proof}
Without loss of generality we may assume that $M$ and $M'$ are open
balls in $\mathbf C^n$ and $\mathbf C^{n'}$, endowed with torsion-free
complex connections and a Hermitian metric $g$ on $M$. Denote by
$v^i_{\bar\gamma}$ and $w^i_{\bar\gamma}$ the real and imaginary parts
of $\partial_{\bar\gamma}u^i$. By the discussion in Ex.~\ref{HerHME}
the coordinate functions $u^i$ satisfy the equation
$$
Lu^i+g^{\alpha\bar\beta}\Gamma^{\prime~i}_{jk}(\partial_\alpha u^j)
(\partial_{\bar\beta}u^k)=0,
$$
where $L$ denotes the holomorphic Laplacian $g^{\alpha\bar\beta}
\partial_{\alpha}\partial_{\bar\beta}$. Applying to this equation
$\partial_{\bar\gamma}$ and using the fact that $L$ is a real operator
(i.e. it maps real functions to real functions), we obtain, after
elementary transformations, that on any compact subset $K\subset M$
the following inequalities hold
$$
\lvert{Lv^i_{\bar\gamma}}\rvert^2\leqslant C_K\sum_{j,\bar\mu}
\left(\lvert{d v^j_{\bar\mu}}\rvert^2+\lvert{d w^j_{\bar\mu}}\rvert^2+
\lvert{v^j_{\bar\mu}}\rvert^2+\lvert{w^j_{\bar\mu}}\rvert^2\right),
$$
$$
\lvert{Lw^i_{\bar\gamma}}\rvert^2\leqslant C_K\sum_{j,\bar\mu}
\left(\lvert{d v^j_{\bar\mu}}\rvert^2+\lvert{d w^j_{\bar\mu}}\rvert^2+
\lvert{v^j_{\bar\mu}}\rvert^2+\lvert{w^j_{\bar\mu}}\rvert^2\right).
$$
Now Aronszajn's theorem~\cite{Ar} applies to show that if
$v^i_{\bar\gamma}$ and $w^i_{\bar\gamma}$ vanish on an open subset of
$K$, then they vanish on all interior of $K$. Since the compact set
$K$ is arbitrary, we are done.
\end{proof}

\end{document}